\documentclass[a4paper,12pt]{amsart}

\usepackage[utf8]{inputenc}

\usepackage[utf8]{inputenc}

\usepackage{amsmath,amssymb,amsxtra,comment,graphicx,psfrag}
\usepackage{bm,mathrsfs}
\usepackage{mathtools}
\usepackage{xspace}
\usepackage{stmaryrd}
\usepackage{color}
\usepackage{lineno}

\usepackage[backend=biber,style=numeric, firstinits=true, maxalphanames=4, maxnames=99,minnames=99]{biblatex}
\renewbibmacro{in:}{}
\DeclareSourcemap{
  \maps[datatype=bibtex]{
    \map{
      \step[fieldsource=journal, match={Numerische Mathematik}, replace={Numer. Math.} ]
    }
     \map{
      \step[fieldsource=journal, 
     match={SIAM Journal on Numerical Analysis}, replace={SIAM J. Numer. Anal.} ]
    }
    \map{
      \step[fieldsource=journal, 
     match={Journal of Computational Physics}, replace={J. Comput. Phys.} ]
    } 
    \map{
      \step[fieldsource=journal, 
     match={Acta numerica}, replace={Acta Numer.} ]
    }
  }
}

\usepackage{amsmath,amssymb,amsxtra,comment,graphicx,psfrag}
\usepackage{bm,mathrsfs}
\usepackage{mathtools}
\usepackage{xspace}
\usepackage{stmaryrd}
\usepackage{color}
\usepackage{lineno}

\usepackage{enumerate,comment,graphicx,psfrag}

\usepackage{hhline}

\usepackage{bookmark}
\usepackage{nameref}

\usepackage{verbatim}
\usepackage{tabls}

\usepackage{fancyhdr}

\usepackage{pstricks,pst-plot,pst-func}
\usepackage{pspicture}
\usepackage{curves}

\usepackage{bbm}

\def\nst2{\| _*} 
 
\def\a12{A_h ^{1/2} } 

\def\tr|{|\! |\! |}

\def\R {{\mathbb R}}

\def\E{{\mathcal{E}}}

\def \a{\alpha }

\def\T_h{{{\mathcal T}_h}}

\def\<{{\langle }}
\def\>{{\rangle }}

\def\S{{\mathcal{S}}}

\DeclarePairedDelimiter{\jumpop}{ \llbracket }{ \rrbracket}
\DeclareSymbolFont{matha}{OML}{txmi}{m}{it}
\DeclareMathSymbol{\varv}{\mathbf}{matha}{118}

\newcommand{\wv}{\widehat{v}}

\newcommand{\rj}{\mathrm{j}}
\newcommand{\rb}{\mathrm{b}}

\theoremstyle{definition}

\newcommand\restr[2]{{
  \left.\kern-\nulldelimiterspace 
  #1 
  \vphantom{\big|} 
  \right|_{#2} 
  }}

\NewDocumentCommand{\dgal}{sO{}m}{%
  \IfBooleanTF{#1}
    {\dgalext{#3}}
    {\dgalx[#2]{#3}}%
}

\NewDocumentCommand{\dgalext}{m}{%
  \sbox0{%
    \mathsurround=0pt 
    $\left\{\vphantom{#1}\right.\kern-\nulldelimiterspace$%
  }%
  \sbox2{\{}%
  \ifdim\ht0=\ht2
    \{\kern-.625\wd2 \{#1\}\kern-.625\wd2 \}%
  \else
    \left\{\kern-.7\wd0\left\{#1\right\}\kern-.7\wd0\right\}%
  \fi
}

\NewDocumentCommand{\dgalx}{om}{%
  \sbox0{\mathsurround=0pt$#1\{$}%
  \sbox2{\{}%
  \ifdim\ht0=\ht2
    \{\kern-.625\wd2 \{#2\}\kern-.625\wd2 \}%
  \else
    \mathopen{#1\{\kern-.7\wd0 #1\{}
    #2
    \mathclose{#1\}\kern-.7\wd0 #1\}}
  \fi
}

\renewcommand{\R}{{\mathbb{R}}}

\renewcommand{\T}{\mathbb{T}^d}

\begin{document}
\theoremstyle{plain}
\newtheorem{theorem}{Theorem}[section]
\newtheorem{lemma}{Lemma}[section]
\newtheorem{proposition}{Proposition}[section]
\newtheorem{corollary}{Corollary}[section]
\newtheorem{problem}{Problem}[section]

\theoremstyle{definition}
\newtheorem{definition}[theorem]{Definition}

\newtheorem{example}[theorem]{Example}

\newtheorem{remark}{Remark}[section]
\newtheorem{remarks}[remark]{Remarks}
\newtheorem{note}{Note}
\newtheorem{case}{Case}

\newcommand{\tristan}[1]{%
  \par\noindent%
  \begingroup
  \color{orange}%
  \textbf{Tristan says:}~#1%
  \par\endgroup
}

\numberwithin{equation}{section}
\numberwithin{table}{section}





\title[PINNS - DG methods]
{PINN-DG: Residual neural network methods  trained with Finite Elements}

\author[Georgios Grekas]{Georgios Grekas}
\address{
Computer, Electrical, Mathematical Sciences \& Engineering Division, King Abdullah University of Science and Technology (KAUST), Thuwal, Saudi Arabia} 
\email {\href{mailto:grekas.g@gmail.com}{grekas.g{\it @\,}gmail.com}}

\author[Charalambos G. Makridakis]{Charalambos G. Makridakis}
\address{
Institute of Applied and Computational Mathematics, 
FORTH, 700$\,$13 Heraklion, Crete, Greece, and MPS, University of  Sussex, Brighton BN1 9QH, United Kingdom} 
\email {\href{mailto:C.G.Makridakis@iacm.forth.gr}{C.G.Makridakis{\it @\,}iacm.forth.gr}}

\author[Tristan Pryer]{Tristan Pryer}
\address{Department of Mathematical Sciences, University of Bath, Claverton Down, Bath BA2 7AY, UK}
\email {\href{mailto:tmp38@bath.ac.uk}{tmp38{\it @\,}bath.ac.uk}} 

\date{\today}


\begin{abstract}
Over the past few years, neural network methods have evolved in various directions for approximating partial differential equations (PDEs). A promising new development is the integration of neural networks with classical numerical techniques such as finite elements and finite differences. In this paper, we introduce a new class of Physics-Informed Neural Networks (PINNs) trained using discontinuous Galerkin finite element methods. Unlike standard collocation-based PINNs that rely on pointwise gradient evaluations and Monte Carlo quadrature, our approach computes the loss functional using finite element interpolation and integration. This avoids costly pointwise derivative computations, particularly advantageous for elliptic PDEs requiring second-order derivatives, and inherits key stability and accuracy benefits from the finite element framework. We present a convergence analysis based on variational arguments and support our theoretical findings with numerical experiments that demonstrate improved efficiency and robustness.
\end{abstract}

\maketitle




\section{Introduction and method formulation}\label{Se:1}

\subsection*{The model problem}\label{SSe:1.1}

In recent years, neural network methods have been developed in a variety of directions for approximating partial differential equations (PDEs). A particularly promising line of work involves integrating neural networks with established numerical methods such as finite element and finite difference schemes. The goal is to develop hybrid techniques that combine the strengths of both paradigms. 

In this paper, we propose a novel formulation of Physics-Informed Neural Networks (PINNs) that incorporates finite element spaces into the training process. Specifically, rather than using standard quadrature rules that lead to collocation-based loss evaluation, we adopt a finite element-based approach that enables efficient computation of otherwise intractable loss functionals. This strategy eliminates the need for pointwise gradient evaluations, particularly advantageous in elliptic problems, where second-order derivatives are required, and inherits several desirable features of the finite element method, stability, flexibility, and scalability in complex geometries.

To illustrate the method, we begin with the classical Poisson problem:
\begin{equation}\label{ivp-AC}
\left \{
\begin{alignedat}{3}
&-\varDelta u  = f\quad &&\text{in}\,\,&& \Omega \\[2pt]
&\quad  u =0 \quad &&\text{on}\,\,&&\partial\Omega, \\[2pt]
\end{alignedat}
\right. 
\end{equation}
where $\Omega \subset \R^d$, for $d=1,2,3$.

We define the associated energy functional
\begin{align}\label{energy_fun}
\mathcal{E}[u] = \int_\Omega \left| \varDelta u + f \right|^2 dx,
\end{align}
and consider weak solutions of \eqref{ivp-AC} as minimisers of $\mathcal{E}$ over the space 
$V(\Omega) = \{ w \in H^1_0(\Omega) : \varDelta w \in L^2(\Omega) \}$. This variational characterisation motivates a PINN approach, where one seeks minimisers of \eqref{energy_fun} over a neural network ansatz space.

Here we consider functions $u_\theta$ defined through neural networks and we are interested in approximating minimisers of 
\eqref{energy_fun} training the network parameters which are represented collectively by the index $\theta$.
Choosing regular enough activation functions, a discretized version of the energy functional in \eqref{energy_fun} is minimised. In fact,
the standard approach  \emph{training through Monte-Carlo quadrature} yielding a collocation type loss functional: 
To be more precise we    consider a collection $ X_1,X_2,\ldots$ of i.i.d.\ $\varOmega$-valued random variables, defined on an appropriate probability space 
representing a random choice of points in $\varOmega\, .$
If we approximate the energy functional by Monte Carlo integration we shall need to evaluate the integrant at random points  $X_i(\omega )\in \varOmega,$
 $\omega $   being a fixed instance:  
 \begin{equation}\label{prob_E}
\mathcal{E}_{N, \omega }( u )  =  \frac 1N  \sum _{i=1 } ^N    \, \left| ( \varDelta u + f )\, (X_i(\omega)\right|^2   \, .  \end{equation}
For the sake of simplicity in the exposition, we disregard the boundary terms that are typically included in a weak manner in the loss functional.
Standard PINN methods rely on the minimization of functionals of the form \eqref{prob_E} over parametrized neural network spaces of a given architecture.
Point evaluations of the operator at the random points are particularly straightforward and yield a collocation-type loss. The primary advantage of such training methods is that they scale reasonably in high dimensions. However, collocation-type methods admit less natural discrete representations, and their implementation requires computationally demanding algorithms such as back-propagation for each operator evaluation. Our objective in this work is to propose alternative training methods that utilise finite elements to approximate the continuous loss. 

\subsubsection*{Finite Elements and Neural Network Functions}
We adopt standard notation for finite element functions defined on a triangulation $T_h$ of a domain $\varOmega$, with local mesh size denoted by $h = h(x)$. For each $K \in T_h$, let $\mathbb{P}_q(K)$ denote the set of polynomials of degree at most $q$. We define the $C^1$-conforming finite element space
\begin{equation}\label{Sh_def_arg}
 \S_{h, 1}(\varOmega) := \left\{ v \in C^1(\bar \varOmega) \,:\, \restr{v}{K} \in \mathbb{P}_q(K), \, K \in T_h \right\}.
\end{equation}
These spaces can be realised, for example, using Argyris elements and satisfy $\S_{h,1}(\varOmega) \subset H^2(\varOmega)$; see \cite[Chapter 3]{BrennerSFE}. We also consider the simpler $C^0$-conforming spaces of continuous piecewise polynomials,
\begin{equation}\label{Sh_def_C0}
 \S_{h, 0}(\varOmega) := \left\{ v \in C^0(\bar \varOmega) \,:\, \restr{v}{K} \in \mathbb{P}_q(K), \, K \in T_h \right\}.
\end{equation}

For a given space $\S_{h,1}(\varOmega)$, the interpolation operator $I_{\S_{h,1}} : C^1(\bar \varOmega) \to \S_{h,1}(\varOmega)$ is defined using the standard basis associated with the finite element space. When the target function $v$ is sufficiently smooth, we may use this interpolant to define a discrete version of the energy functional $\E(v)$ by
\begin{equation}
\label{E_hFE}
\mathcal{E}_{\S_{h,1}}(v) = \int_{\varOmega} \left| \varDelta I_{\S_{h,1}} v + I_{\S_{h,1}} f \right|^2 \, \mathrm{d}x.
\end{equation}
Minimisers of $\mathcal{E}_{\S_{h,1}}(v)$ will be sought over neural network functions $v = v_\theta$, where $\theta$ denotes the network parameters.

This approach, training neural networks via finite element interpolation, has been recently proposed in \cite{GrekasM_2025}. The idea is to evaluate the loss functional by interpolating the network output into a conforming finite element space, typically a subspace of $H^2(\Omega)$. However, implementing $C^1$-conforming spaces such as $\S_{h,1}(\varOmega)$ can be challenging in practice, especially in three dimensions.

The methods proposed in this work employ the Discontinuous Galerkin framework to approximate the continuous energy functional $\mathcal{E}$. For ease of implementation, we adopt finite element spaces with lower regularity, specifically the standard $C^0$ finite element space $\S_{h,0}(\varOmega)$. An interpolation operator $I_{\S_{h,0}} : C^0(\bar\varOmega) \to \S_{h,0}(\varOmega)$ is used to define the discrete energy functional:
\begin{equation}
\begin{aligned}\label{E_h_intro}
\mathcal{E}_{\S_{h,0}}(v) &= \sum_{K \in T_h} \int_K \left| \varDelta I_{\S_{h,0}} v + I_{\S_{h,0}} f \right|^2 \, dx 
\\
&\qquad + \operatorname{cons}(I_{\S_{h,0}} v, f) + \alpha\, \operatorname{pen}(I_{\S_{h,0}} v, f).
\end{aligned}
\end{equation}
The additional terms $\operatorname{cons}(I_{\S_{h,0}} v, f)$ and $\operatorname{pen}(I_{\S_{h,0}} v, f)$ denote consistency and penalty contributions, respectively. These typically involve jumps and averages across element interfaces, as is standard in Discontinuous Galerkin methods. Precise definitions are given in \eqref{E_h}. For related formulations in the finite element literature, see e.g.~\cite{brenner2005c, engel2002continuous, makridakis2014atomistic, pryer2014discontinuous, grekas2022approximations}.

\subsection*{Contribution, results and bibliography}
\label{sec:G_convergence}

\subsubsection*{\it DG-PINNs.} 
We propose a novel approach that connects Discontinuous Galerkin finite elements with Physics-Informed Neural Networks through \emph{training}, by employing finite element spaces to evaluate the integrals appearing in the loss functional. 

Residual-based loss formulations involve the differential operator directly, rather than its variational form, and thus require higher regularity. However, standard $C^0$ finite element functions typically have discontinuous gradients. To ensure consistency and stability, we introduce appropriate gradient jump terms across element interfaces, following standard techniques in the Discontinuous Galerkin framework~\cite{brenner2005c}.

While previous approaches have typically relied on quadrature- or Monte Carlo-based collocation methods, our method preserves the simplicity of standard neural network formulations for PDEs while incorporating well-established finite element tools (and software) for loss computation. This hybrid approach leads to improved accuracy and computational efficiency.

The resulting method is broadly applicable and can be extended to a wide class of PDEs, since least-squares loss formulations and Physics-Informed Neural Networks are in principle suited for both linear and nonlinear problems. However, the choice of finite element spaces and corresponding discrete loss formulations is non-trivial and must reflect the stability properties of the continuous problem. These methods are especially effective when appropriate finite element spaces can be constructed efficiently.

\subsubsection*{\it Stability and Convergence.} 
The theoretical foundations of our method are detailed in Section 3. To establish stability and convergence—as advocated in\cite{GGM_pinn_2023}—we derive first discrete coercivity estimates for the loss functional. Then, assuming abstract approximability properties of the discrete spaces, we prove convergence by employing the   $\liminf$–$\limsup$ framework of DeGiorgi (see Section2.3.4 of~\cite{DeGiorgi_sel_papers:2013}), which is closely related to $\Gamma$-convergence, a standard tool in variational analysis.

To this end, we combine finite element techniques with arguments from the calculus of variations, and we establish both the stability of the method (see Proposition~\ref{Prop:EquicoercivityofE(2)}) and its convergence (see Theorem~\ref{Thrm:Gamma_funct}). 

\subsubsection*{\it Numerical Experiments.}
In Section~\ref{CE}, we present numerical experiments that demonstrate the accuracy, robustness and efficiency of the proposed method. The integration of finite element discretisations into the training process enables consistent approximations across different network architectures and domain geometries, including cases with singularities. Compared to collocation-based PINNs, the proposed approach achieves comparable or improved accuracy at a significantly reduced computational cost and memory usage. The experiments also highlight the stabilising role of gradient jump terms, especially for deeper architectures and under-resolved training meshes.

\subsubsection*{\it Literature Review.}
The approximation of partial differential equations (PDEs) using neural networks has been extensively studied in recent years. The main approaches include \cite{e2017deep}, \cite{Karniadakis:pinn:orig:2019}, and \cite{kharazmi2019variational}. Physics-Informed Neural Networks (PINNs), introduced in~\cite{Karniadakis:pinn:orig:2019}, incorporate physical modeling into the learning process by embedding residual minimization into the loss functional. Residual-based methods   have also been explored in  e.g.,  \cite{Lagaris_1998}, \cite{Berg_2018}, \cite{Raissi_2018}, and \cite{SSpiliopoulos:2018}. These serve as the starting point for the methodology developed in this paper. While our focus is on elliptic problems, the PINN framework can in principle be extended to a wide range of PDEs. An approach which is still residual based, uses the     optimisation and related toolboxes of modern machine learning algorithms but avoids the use of neural network spaces was suggested in \cite{Koumoutsakos_ODIL_2023}.   For time-dependent problems, time-discrete training strategies are investigated in \cite{AMS25}. Alternative perspectives on neural network approaches to differential equations and related problems can be found in \cite{kevr_1992discrete}, \cite{Xu}, \cite{chen2022deep}, \cite{georgoulis2023discrete}, \cite{Grohs:space_time:2023}, \cite{MPP_24_bc}, and \cite{MPP_24_dUz}. 

The use of finite elements in the training of deep Ritz methods was introduced and analysed in \cite{GrekasM_2025}.  Methods that incorporate finite elements into neural network approximations through the Variational PINN approach  were studied in \cite{B_Canuto_P_Vpinn_quadrature_2022} and \cite{Badia_2024}; see also \cite{HybridNN_Fang2022} and \cite{meethal2022finiteelementmethodenhancedneural}.

Convergence and error analysis for neural network-based PDE solvers has been pursued in several directions. Works such as \cite{SSpiliopoulos:2018}, \cite{Mishra_dispersive:2021}, \cite{Karniadakis:pinn:conv:2019}, \cite{shin2020error}, \cite{Mishra:pinn:inv:2022}, \cite{hong2022priori}, and \cite{Mishra_gen_err_pinn:2023} establish estimates where bounds often depend on the properties of the discrete minimisers and their derivatives. The work \cite{hong2022priori} focuses on deterministic training in spaces where neural networks are constructed to ensure uniform control of high-order derivatives. The deep Ritz method was studied via $\Gamma$-convergence in \cite{muller2020deep} without explicit training. Recent results in \cite{loulakis2023new} apply the $\liminf$–$\limsup$ framework to analyse the convergence of discrete minimisers in probabilistically trained machine learning schemes. The use of this variational framework for deterministic PINNs was first proposed in \cite{GGM_pinn_2023}, with applications to elliptic and parabolic problems. Related developments for  finite element methods utilising  complex discrete energy functionals appear in \cite{bartels2017bilayer} and \cite{grekas2022approximations}.

\section{Preliminaries}
\subsection{Finite Element Spaces}

We shall use standard notation for Sobolev spaces
$W^{s, p} (\Omega)$, consisting of functions with weak derivatives up to order $s$ in $L^p(\Omega)$, defined on a domain $\Omega$.
The corresponding norm is denoted by $\| \cdot \|_{W^{s, p} (\Omega)}$, the seminorm by $| \cdot |_{W^{s, p} (\Omega)}$, and the $L^p$-norm by
$\| \cdot \|_{L^p (\Omega)}$.

We assume that $\Omega$ is a domain with sufficiently smooth boundary $\partial \Omega$, and let $T_h$ be a triangulation of $\Omega$ with mesh size $h$. To avoid additional complications in the subsequent analysis, we assume that any element $K$ with an edge $e$ on the boundary satisfies $e \subset \partial \Omega$. Furthermore, the partition is assumed to be shape-regular; see \cite{BrennerSFE}.

Next, we define the standard space of continuous piecewise polynomial functions as
\begin{equation}\label{Sh_def}
\S^q_h(\Omega):=\left\{v \in C^0(\bar \Omega)\,:\,\restr{v}{K}  \in \mathbb{P}_q(K),\; K \in T_h \right\},
\end{equation}
where $\mathbb{P}_q(K)$ denotes the set of polynomials of total degree at most $q \in \mathbb{N}$ on element $K$.

We also consider the discontinuous finite-dimensional space
\begin{equation}\label{tSh_def}
\tilde{\S}^q_h(\Omega):=\left\{v \in L^2(\Omega)\,:\,\restr{v}{K} \in \mathbb{P}_q(K),\; K \in T_h \right\}.
\end{equation}

Let $E_h$ denote the set of mesh edges. We partition $E_h$ into boundary edges $E^b_h = E_h \cap \partial \Omega$ and internal edges $E^i_h = E_h \setminus E^b_h$.

The trace of functions in $\tilde{\S}^q_h(\Omega)$ lies in the space $T(E_h):= \prod_{e \in E_h} \mathbb{P}_q(e)$. We define the average and jump operators by:
\begin{equation}
\begin{aligned}
\dgal{ \cdot } :&\ T(E_h) \to L^p(E_h), \\
\dgal{ w } :=&
             \frac{1}{2} \left( \restr{w}{K_{e^+}} + \restr{w}{K_{e^-}} \right), \quad \text{for } e \in E^i_h.
\end{aligned}
\label{average_operator}
\end{equation}

\begin{equation}
\begin{aligned}
\jumpop{ \cdot } :&\ T(E_h) \to L^p(E_h), \\
\jumpop{ v } :=&
   \left\{
         \begin{array}{ll}
              \restr{v}{K_{e^+}} - \restr{v}{K_{e^-}}, & \text{for } e \in E^i_h, \\
              v, & \text{for } e \in E^b_h.
         \end{array}
     \right.
\end{aligned}
\label{jump_operator}
\end{equation}

Note that the jump of the normal derivative satisfies
\[
\jumpop{ \nabla v \cdot n_e } = \restr{\nabla v \cdot n_{e^+}}{K_{e^+}} + \restr{\nabla v \cdot n_{e^-}}{K_{e^-}},
\]
where $K_{e^+}$ and $K_{e^-}$ are the elements sharing the internal edge $e$, and $n_{e^+}, n_{e^-}$ are the corresponding outward unit normals.

\subsubsection*{Approximation properties}

For convenience, we briefly recall some standard approximation properties of finite element and neural network spaces.

We begin by defining the standard interpolation operator $I_{\S^q_h}$ and summarising its approximation behaviour. Let $\Phi_{\mathcal{Z}}$ denote the Lagrangian basis of $\S^q_h(\Omega)$, where $\mathcal{Z}$ is the set of degrees of freedom. The interpolant $I_{\S^q_h} : C^0(\bar{\Omega}) \rightarrow \S^q_h(\Omega)$ is defined as
\begin{align}
I_{\S^q_h} w(x) = \sum_{z \in \mathcal{Z}} w(z) \Phi_z(x).
\end{align}

If $u \in H^s(\Omega)$ and $q \ge \lceil s \rceil - 1$, then standard interpolation error estimates yield
\begin{align}
|u - I_{\S^q_h} u|_{H^m(K)} \le c\, h_K^{s-m} |u|_{H^s(K)}, \quad 0 \le m \le s,
\label{error_est}
\end{align}
where $h_K$ denotes the diameter of an element $K \in T_h$.

For edges $e \subset \partial K$, the following trace estimate holds:
\begin{align}
|u - I_{\S^q_h} u|_{H^m(e)} \le c\, h_K^{s - m - 1/2} |u|_{H^s(K)}.
\label{ineq:error_on_e}
\end{align}

From these estimates, we obtain the following proposition.

\begin{proposition}
\label{prop:jump_vanish}
Let $v \in H^s(\Omega)$ with $s > 2$, and define $v_h = I_{\S^q_h} v$ and $f_h = I_{\S^q_h} f$ for some $q \ge \lceil s \rceil - 1$. Then the following estimates hold:
\begin{align}
\sum_{e \in E_h^i} \frac{1}{h_e} \int_e |\jumpop{\nabla v_h}|^2 
\lesssim h^{2s - 4} |v|^2_{H^s(\Omega)},
\label{eq:jumps_bnd1}
\end{align}
and
\begin{equation}
\begin{aligned}
\left| \sum_{e \in E_h^i} \int_e \dgal{ \varDelta v_h + f_h} \jumpop{\nabla v_h \cdot n} \right|
&\lesssim h^{s - 2} |v|_{H^2(\Omega)} |v|_{H^s(\Omega)} \\
&\quad + (h + 1) \|f\|_{H^1(\Omega)} h^{s - 2} |v|_{H^s(\Omega)}.
\end{aligned}
\label{eq:cons_bnd1}
\end{equation}
\end{proposition}

\begin{proof}
Estimate \eqref{eq:jumps_bnd1}, and the first term in \eqref{eq:cons_bnd1} involving $\varDelta v_h$, follow by adapting the proof of \cite[Proposition 5.2]{grekas2022approximations}, replacing the Hessian $D^2 v_h$ with the Laplacian $\varDelta v_h$.

To bound the term involving $f_h$, we apply the Cauchy–Schwarz inequality and an inverse inequality of the form
\[
\|u_h\|_{L^2(e)} \lesssim \frac{1}{h_e^{1/2}} \|u_h\|_{L^2(K_e)},
\]
where $e$ is an edge, $K_e$ are the elements containing $e$, and $h_e = |e|$. Combined with the interpolation estimate \eqref{error_est} and the assumption $f \in C^0(\Omega)$, we obtain
\begin{equation}
\begin{aligned}
\sum_{e \in E_h^i} \int_e \left| f_h \jumpop{\nabla v_h \cdot n} \right|
&\le \left( \sum_{e \in E_h^i} \int_e |f_h|^2 \right)^{1/2}
     \left( \sum_{e \in E_h^i} \int_e |\jumpop{\nabla v_h \cdot n}|^2 \right)^{1/2} \\
&\lesssim \left( \sum_{e \in E_h^i} \frac{1}{h_e} \int_{K_e} |f_h|^2 \right)^{1/2}
     \left( \sum_{e \in E_h^i} \frac{1}{h_e} \int_e |\jumpop{\nabla v_h}|^2 \right)^{1/2} \\
&\lesssim \|f_h\|_{L^2(\Omega)} \left( \sum_{e \in E_h^i} \frac{1}{h_e} \int_e |\jumpop{\nabla v_h}|^2 \right)^{1/2} \\
&\lesssim (h + 1) \|f\|_{H^1(\Omega)} h^{s - 2} |v|_{H^s(\Omega)}.
\end{aligned}
\end{equation}
\end{proof}

\subsection{Discrete spaces generated by Residual Neural Networks}\label{Se:1NN}

We partially follow the exposition of \cite{GGM_pinn_2023, BKMM}, considering functions 
$u_\theta$ defined via neural networks. Specifically, we employ residual neural networks.

A residual neural network maps each point $x \in \Omega$ to a scalar $u_\theta(x) \in \R$ via
\begin{equation}\label{C_L}
	u_\theta(x) = C_o \circ \text{bl}_{L} \circ \text{bl}_{L-1} \cdots \circ \text{bl}_1 \circ \sigma \circ C_i(x), \quad \forall x \in \Omega \subset \R^d.
\end{equation}

The composition
\begin{equation}
 \mathcal{C}_L := C_o \circ \text{bl}_L \circ \text{bl}_{L-1} \cdots \circ \text{bl}_1 \circ \sigma \circ C_i
\end{equation}
defines a map $\mathcal{C}_L : \R^m \to \R^{m'}$. In our setting, $m = d$ and $m' = 1$, so that $u_\theta(\cdot) = \mathcal{C}_L(\cdot)$.

Each map $\mathcal{C}_L$ is a residual neural network with $L$ blocks and activation function $\sigma$. It is characterised by the input and output layers $C_i$, $C_o$, which are affine maps:
\begin{equation}\label{C_k}
\begin{aligned}
	C_i(x) &= W_i x + b_i, &\quad &W_i \in \R^{N \times d},\; b_i \in \R^N, \\
	C_o(z) &= w_o \cdot z + b_o, &\quad &w_o \in \R^N,\; b_o \in \R,
\end{aligned}
\end{equation}
and intermediate blocks defined by
\begin{equation}\label{eq:blocks}
\text{bl}_k(z) = \sigma \left( \left( W_{2k} \sigma(W_{1k} z + b_{1k}) + b_{2k} \right) + z \right),
\end{equation}
where $z, b_{jk} \in \R^N$, $W_{jk} \in \R^{N \times N}$ for $j = 1, 2$ and $k = 1, \dots, L$. Here $\sigma(y)$ denotes the vector obtained by applying $\sigma$ componentwise: $\sigma(y)_i = \sigma(y_i)$.

The index $\theta$ collectively denotes all network parameters in $\mathcal{C}_L$. The set of all such networks with fixed structure (fixed $L, N, d$) is denoted by $\mathcal{N}$. We then define the corresponding space of neural network functions:
\begin{equation}
	V_{\mathcal{N}} = \left\{ u_\theta : \Omega \to \R \;\middle|\; u_\theta(x) = \mathcal{C}_L(x) \text{ for some } \mathcal{C}_L \in \mathcal{N} \right\}.
\end{equation}

It is important to note that $V_{\mathcal{N}}$ is \emph{not a linear space}. In general, if $u_\theta, u_{\tilde \theta} \in V_{\mathcal{N}}$, there does not exist a network $\hat{\mathcal{C}}_L \in \mathcal{N}$ such that $u_\theta + u_{\tilde \theta} = \hat{\mathcal{C}}_L$.

However, there is a one-to-one correspondence between parameters $\theta$ and functions $u_\theta \in V_{\mathcal{N}}$, so we may identify
\begin{equation}
	\theta \mapsto u_\theta \in V_{\mathcal{N}}.
\end{equation}
The corresponding parameter space is
\begin{equation}
	\Theta = \left\{ \theta \in \R^{\dim(\mathcal{N})} \;\middle|\; u_\theta \in V_{\mathcal{N}} \right\},
\end{equation}
which is a linear subspace of $\R^{\dim(\mathcal{N})}$.

Since the main focus of this work is the design of stable and consistent loss functionals it will be useful for the forthcoming analysis  to assume that the neural network spaces satisfy abstract approximation properties as follows: Given a function $w \in W^{s,p}(\Omega)$ with $1 \le p \le \infty$, we now describe the approximation properties of neural network spaces. For each $\ell \in \mathbb{N}$, we associate a neural network space $V_\ell \subset V_{\mathcal{N}}$ such that for $m \ge s + 1$, there exists $w_\ell \in V_\ell$ satisfying
\begin{equation}\label{w_ell_7_hSm}
\|w_\ell - w\|_{W^{s,p}(\Omega)} \le \tilde{\beta}_\ell^{[m, s, p]} \, |w|_{W^{m,p}(\Omega)}, \quad \text{with} \quad \tilde{\beta}_\ell^{[m, s, p]} \to 0 \quad \text{as } \ell \to \infty.
\end{equation}
These bounds may hold only for $w=u,$ $u$ being the exact solution of the PDE.  These assumptions are compatible with  approximation results for neural networks which can be found in \cite{Xu}, \cite{Dahmen_Grohs_DeVore:specialissueDNN:2022}, \cite{Schwab_DNN_constr_approx:2022}, \cite{Schwab_DNN_highD_analystic:2023}, \cite{Mishra:appr:rough:2022}, \cite{Grohs_Petersen_Review:2023}, and references therein.

\subsection{Motivation of the method}
\label{sec:method_motivation}

We seek functions $u_\theta \in V_\mathcal{N}$ that approximate minimisers of the energy defined in \eqref{energy_fun}. When $u_\theta \in H^2(\Omega)$, we may employ a conforming finite element space such that $I_{\S^q_h} u_\theta \in C^1(\Omega)$. This allows us to directly apply the training approach introduced in \cite{GrekasM_2025}, where $u_\theta$ is substituted by $I_{\S^q_h} u_\theta$ in the energy functional \eqref{energy_fun}. However, it is well known that constructing such regular finite element spaces is complicated, especially in three dimensions.

For implementation simplicity, we instead choose $C^0(\Omega)$ elements, so that $I_{\S^q_h} u_\theta \in H^1(\Omega)$, and continuity across element interfaces is imposed weakly through the discontinuous Galerkin framework.

To motivate the design of the method, consider the simplified setting where the energy functional has the form
\begin{equation}
\mathcal{E}(u) = \int_\Omega g(\varDelta u),
\end{equation}
for some function $g : \R \to \R$, and where $\mathcal{E}$ attains a minimum over a Banach space $X$. Let $w$ denote a minimiser. Then, for all $z \in X$,
\begin{equation} \label{mot_variation}
\frac{d}{d\varepsilon} \mathcal{E}(w + \varepsilon z) = \int_\Omega g'(\varDelta w) \varDelta z = 0.
\end{equation}

For further simplification, let us assume that
\[
X = \left\{ \hat{p} + v : \hat{p} \in \mathbb{P}_2(\Omega),\; v \in H^2(\Omega) \text{ periodic} \right\},
\]
where $\mathbb{P}_2(\Omega)$ denotes the space of degree-two polynomials on $\Omega$. Note that any $\hat{p} \in \mathbb{P}_2(\Omega)$ is a critical point satisfying \eqref{mot_variation}. Indeed, for any periodic $v \in H^2(\Omega)$,
\begin{equation} \label{eq:polynom_property}
\int_\Omega g'(\varDelta \hat{p}) \varDelta v = g'(\varDelta \hat{p}) \int_{\partial \Omega} v_{,i} n_i = 0,
\end{equation}
using integration by parts and the summation convention.

We aim for our numerical method to retain this key property, even when test functions have lower regularity. To illustrate the issue, let $\Omega = [0,1] \times [0,1]$ and divide it into two triangles $K^+$ and $K^-$ sharing a common edge $e = K^+ \cap K^-$. If a test function $v$ is allowed to be discontinuous at the interface $e$, then
\begin{equation}
\begin{aligned}
\frac{d}{d\varepsilon} \mathcal{E}(\hat{p} + \varepsilon v) 
&= \int_{K^+} g'(\varDelta \hat{p}) \varDelta v + \int_{K^-} g'(\varDelta \hat{p}) \varDelta v \\
&= g'(\varDelta \hat{p}) \int_{\partial K^+} v_{,i} n^+_i + g'(\varDelta \hat{p}) \int_{\partial K^-} v_{,i} n^-_i \\
&= g'(\varDelta \hat{p}) \left( \int_{\partial \Omega} v_{,i} n_i + \int_e v^+_{,i} n^+_i + v^-_{,i} n^-_i \right) \\
&= g'(\varDelta \hat{p}) \int_e \jumpop{\nabla v \cdot n},
\end{aligned}
\end{equation}
where $v^+_{,i}, v^-_{,i}$ denote the restrictions of $v$ to $K^+$ and $K^-$, and $n^+, n^-$ are the outward unit normals on $\partial K^+$ and $\partial K^-$. The final expression shows that the fundamental property \eqref{eq:polynom_property} is violated due to the discontinuity across $e$.

Motivated by the analysis in \cite{makridakis2014atomistic}, we modify the energy functional to
\begin{equation}
\mathcal{E}_{\text{dg}}(u) = \int_{K^+} g(\varDelta u)\,dx + \int_{K^-} g(\varDelta u)\,dx - \int_e \dgal{g'(\varDelta u)} \jumpop{\nabla u \cdot n}\,ds.
\end{equation}

This discontinuous Galerkin energy functional is now consistent in the sense that
\[
\frac{d}{d\varepsilon} \mathcal{E}_{\text{dg}}(\hat{p} + \varepsilon v) = 0
\]
holds for any periodic test function $v$ that may be discontinuous across the internal edge $e$.

\section{Method formulation and convergence analysis}
\label{sec:G_convergence}

\subsection{The Discrete Energy Functional}

Following the discussion in Section~\ref{sec:method_motivation}, we define the discrete energy functional for $u_\ell \in V_\ell$ as
\begin{equation}\label{E_h}
\begin{aligned}
\tilde{\mathcal{E}}_h(u_\ell) &= \sum_{K \in T_h} \int_K \left| \varDelta I_{\S^q_h} u_\ell + I_{\S^q_h} f \right|^2\,dx \\
&\quad - 2 \sum_{e \in E_h^i} \int_e \dgal{ \varDelta I_{\S^q_h} u_\ell + I_{\S^q_h} f } \jumpop{ \nabla I_{\S^q_h} u_\ell \cdot n } \\
&\quad + \alpha\, \mathrm{pen}(u_\ell).
\end{aligned}
\end{equation}

As is standard in the Discontinuous Galerkin setting, the term $\alpha\, \mathrm{pen}(u_\ell)$ ensures coercivity. It penalises discontinuities across element interfaces and imposes boundary conditions in the spirit of Nitsche's method \cite{nitsche1971variationsprinzip}. Recent DG schemes for nonconvex variational problems \cite{Grekas_2025,grekas2025convergencediscontinuousgalerkinmethods} have proposed general penalty terms; however, due to the convexity of the integrand here, we adopt a simpler form:
\begin{equation}\label{eq:pen}
\mathrm{pen}(u_\ell) = \sum_{e \in E_h^i} \frac{1}{h_e} \int_e \left| \jumpop{ \nabla I_{\S^q_h} u_\ell } \right|^2 + 
\sum_{e \in E_h^b} \frac{1}{h_e} \int_e \left| I_{\S^q_h}(u_\ell - g) \right|^2,
\end{equation}
where $g$ denotes the Dirichlet boundary condition in the original problem \eqref{SSe:1.1}. In what follows, we assume $g = 0$ to simplify the treatment of boundary terms and focus on the core methodology.

An equivalent formulation of \eqref{E_h} may be obtained using a lifting operator, leading to a natural definition of a discrete Laplacian. The $L^p$-boundedness of the lifting operator in terms of the penalty ensures coercivity, and convergence can be proved via $\liminf$ and $\limsup$ arguments, as in \cite{grekas2022approximations}. For simplicity, we adopt this lifting-based formulation.

Let $R_h(\nabla u_h) : T(E_h) \to \tilde{\S}^q_h(\Omega)$ be the lifting operator, defined as in \cite{ten2006discontinuous} by
\begin{equation}
\int_\Omega R_h(\nabla u_h)\, w_h = \sum_{e \in E_h^i} \int_e \dgal{w_h} \jumpop{ \nabla u_h \cdot n }, \quad \forall w_h \in \tilde{\S}^q_h(\Omega),\; u_h \in \S^q_h(\Omega).
\end{equation}

Observe that for $u_{h,\ell} = I_{\S^q_h} u_\ell$ and $f_h = I_{\S^q_h} f$, since $\varDelta u_{h,\ell} + f_h \in \tilde{\S}^q_h(\Omega)$, we have
\begin{equation}
\int_\Omega R_h(\nabla u_{h,\ell})\, (\varDelta u_{h,\ell} + f_h) 
= \sum_{e \in E_h^i} \int_e \dgal{ \varDelta u_{h,\ell} + f_h } \jumpop{ \nabla u_{h,\ell} \cdot n }.
\end{equation}

This motivates the definition of the discrete Laplacian $L_h : \S^q_h(\Omega) \to L^2(\Omega)$ via
\begin{equation}\label{eq:Lh}
L_h(u_h) = \varDelta_h u_h - R_h(\nabla u_h),
\end{equation}
where the piecewise Laplacian is defined elementwise by
\[
\restr{ \varDelta_h u_h }{K} = \varDelta \left( \restr{u_h}{K} \right), \quad \forall K \in T_h.
\]
This construction is analogous to the discrete gradient framework in \cite{buffa2009compact,di2010discrete,pryer2014discontinuous}. Note that while $\varDelta u_h$ is not globally defined in $L^2(\Omega)$, it belongs to $L^2(K)$ locally, and $L_h(u_h) \in L^2(\Omega)$ via the lifting.

We may now rewrite the discrete energy \eqref{E_h} as
\begin{equation}\label{E_h2}
\begin{aligned}
\tilde{\mathcal{E}}_h(u_\ell) 
&= \int_\Omega \left| \varDelta u_{h,\ell} + f_h \right|^2 - 2 \int_\Omega R_h(\nabla u_{h,\ell})\, ( \varDelta u_{h,\ell} + f_h ) + \alpha\, \mathrm{pen}(u_\ell) \\
&= \int_\Omega \left| L_h(u_{h,\ell}) + f_h \right|^2 - \left| R_h(\nabla u_{h,\ell}) \right|^2 + \alpha\, \mathrm{pen}(u_\ell).
\end{aligned}
\end{equation}

To simplify notation, for each $\ell$ we associate a finite element space $\S^q_{h(\ell)}$ with mesh diameter $h(\ell)$ such that $h(\ell) \to 0$ as $\ell \to \infty$. We then write
\begin{equation}\label{E_ell}
\mathcal{E}_\ell(u_\ell) := \tilde{\mathcal{E}}_h(u_\ell).
\end{equation}

In the remainder of the analysis, we assume that the finite element spaces consist of piecewise polynomials of degree at least two.

\subsection{Properties of the lifting operator and discrete Laplacian}

We recall well-known results on the boundedness of the lifting operator (see \cite{brezzi2000discontinuous, buffa2009compact, di2010discrete}) and its consequences for the discrete Laplacian.

Specifically, we use the estimate given in \cite[Lemma 7]{buffa2009compact} and \cite[Lemma 4.34]{di2011mathematical}:
\begin{align}\label{eq:Rhbound}
\int_\Omega \left| R_h(\nabla v_h) \right|^2 \le C_R \sum_{e \in E_h^i} \frac{1}{h_e} \int_e \left| \jumpop{\nabla v_h} \right|^2, \quad \forall v_h \in \S^q_h(\Omega),
\end{align}
where $C_R$ is a constant independent of $h$.

This bound implies the following estimate for the discrete Laplacian:
\begin{align}\label{eq:dLaplace_bound}
\int_\Omega \left| L_h(u_h) \right|^2 \lesssim \sum_{K \in T_h} \int_K \left| \varDelta u_h \right|^2 + \sum_{e \in E_h^i} \frac{1}{h_e} \int_e \left| \jumpop{\nabla u_h} \right|^2, \quad \forall u_h \in \S^q_h(\Omega).
\end{align}

We now state a compactness result, adapted from the proofs of \cite[Theorem 2.2]{di2010discrete} and \cite[Lemma 5.2]{grekas2022approximations}:

\begin{lemma}\label{Lem:dlapl_weak_conv}
Let $u_h \in \S^q_h(\Omega)$ with $u_h \rightharpoonup u$ in $H^1(\Omega)$, where $u \in V(\Omega)$. If
\begin{align}\label{eq:lem_uniformbnd}
\sum_{K \in T_h} \int_K \left| \varDelta u_h \right|^2 + \sum_{e \in E_h^i} \frac{1}{h_e} \int_e \left| \jumpop{\nabla u_h} \right|^2 < C,
\end{align}
uniformly with respect to $h$, then
\begin{align}\label{Lh_wconv}
L_h(u_h) \rightharpoonup \varDelta u \quad \text{in } L^2(\Omega).
\end{align}
\end{lemma}

\begin{proof}
Let $\phi \in C_c^\infty(\Omega)$. By the definition of the discrete Laplacian in \eqref{eq:Lh}, we have
\begin{equation}
\begin{aligned}
\int_\Omega L_h(u_h)\, \phi &= \sum_{K \in T_h} \int_K \varDelta u_h\, \phi - \int_\Omega R_h(\nabla u_h)\, \phi \\
&= -\int_\Omega \nabla u_h \cdot \nabla \phi + \sum_{e \in E_h^i} \int_e \phi\, \jumpop{ \nabla u_h \cdot n } - \int_\Omega R_h(\nabla u_h)\, \phi.
\end{aligned}
\end{equation}

Our goal is to show that the last two terms vanish as $h \to 0$. Let $\phi_h = I_{\tilde{\S}^q_h(\Omega)} \phi$ be the Lagrange interpolant of $\phi$. Using the error estimates \eqref{error_est} and \eqref{ineq:error_on_e}, the lifting bound \eqref{eq:Rhbound}, and the uniform bound \eqref{eq:lem_uniformbnd}, we estimate:
\begin{equation}
\begin{aligned}
&\left| \sum_{e \in E_h^i} \int_e \phi\, \jumpop{ \nabla u_h \cdot n } - \int_\Omega R_h(\nabla u_h)\, \phi \right| \\
&= \left| \sum_{e \in E_h^i} \int_e \phi_h\, \jumpop{ \nabla u_h \cdot n } + \int_e (\phi - \phi_h) \jumpop{ \nabla u_h \cdot n } - \int_\Omega R_h(\nabla u_h)\, \phi_h - \int_\Omega R_h(\nabla u_h)\, (\phi - \phi_h) \right| \\
&\lesssim \sum_{e \in E_h^i} \int_e \left| \phi - \phi_h \right| \left| \jumpop{ \nabla u_h \cdot n } \right| + \int_\Omega \left| R_h(\nabla u_h) \right| \left| \phi - \phi_h \right| \\
&\lesssim h \, |\phi|_{H^1(\Omega)} \left( \sum_{e \in E_h^i} \frac{1}{h_e} \int_e \left| \jumpop{ \nabla u_h } \right|^2 \right)^{1/2} + h\, |\phi|_{H^1(\Omega)} \lesssim h\, |\phi|_{H^1(\Omega)} \to 0,
\end{aligned}
\end{equation}
as $h \to 0$.

Therefore,
\begin{align}
\lim_{h \to 0} \int_\Omega L_h(u_h)\, \phi = - \lim_{h \to 0} \int_\Omega \nabla u_h \cdot \nabla \phi = \int_\Omega \nabla u \cdot \nabla \phi,
\end{align}
since $u_h \rightharpoonup u$ in $H^1(\Omega)$. Hence $L_h(u_h) \rightharpoonup \varDelta u$ in $L^2(\Omega)$.
\end{proof}

\subsection{Stability and convergence}

We begin with the stability of the method, formulated precisely below.

\begin{proposition}[Equicoercivity]\label{Prop:EquicoercivityofE(2)}
Let $\mathcal{E}_\ell$ be defined as in \eqref{E_ell}. Suppose $(v_\ell)$ is a sequence in $V_\ell$ satisfying
\begin{equation}\label{EquicoercivityofEdelta1(2)}
\mathcal{E}_\ell(v_\ell) \leq C,
\end{equation}
for some constant $C > 0$ independent of $\ell$. Then there exists a constant $C_1 > 0$ such that for sufficiently large penalty parameter $\alpha$, we have
\begin{equation}\label{EquicoercivityofEdelta2(2)}
\sum_{K \in T_{h(\ell)}} \int_K |\varDelta \wv_\ell|^2 + \| \wv_\ell \|_{H^1(\Omega)}^2 + \operatorname{pen}(v_\ell) \leq C_1,
\end{equation}
where $\wv_\ell := I_{\S^q_{h(\ell)}} v_\ell$ and $f_\ell := I_{\S^q_{h(\ell)}} f$.
\end{proposition}

\begin{proof}
We first bound the consistency term from below using Cauchy–Schwarz and an inverse inequality:
\begin{equation}\label{cons_bounds}
\begin{aligned}
-2\sum_{e \in E^i_{h(\ell)}} \int_e \dgal{\varDelta \wv_\ell + f_\ell} \jumpop{ \nabla \wv_\ell \cdot n }
\ge & -c_1 \left( \sum_{K \in T_{h(\ell)}} \int_K | \varDelta \wv_\ell + f_\ell |^2 \right)^{1/2}
\\
&\qquad \qquad \times \left( \sum_{e \in E^i_{h(\ell)}} \frac{1}{h_e} \int_e | \jumpop{ \nabla \wv_\ell } |^2 \right)^{1/2} \\
\ge & -\frac{c_1 \delta}{2} \sum_{K \in T_{h(\ell)}} \int_K | \varDelta \wv_\ell + f_\ell |^2
- \frac{c_1}{2\delta} \sum_{e \in E^i_{h(\ell)}} \frac{1}{h_e} \int_e | \jumpop{ \nabla \wv_\ell } |^2,
\end{aligned}
\end{equation}
for any $\delta \in (0,1)$ and a constant $c_1 > 0$. Choosing $\delta$ and $\alpha$ such that $1 - \frac{c_1 \delta}{2} > 0$ and $\alpha - \frac{c_1}{2\delta} > 0$, we obtain
\begin{equation}\label{eq:coerv1stbnd}
\mathcal{E}_\ell(v_\ell) \ge \left(1 - \frac{c_1 \delta}{2}\right) \sum_{K \in T_{h(\ell)}} \int_K | \varDelta \wv_\ell + f_\ell |^2 + \left( \alpha - \frac{c_1}{2\delta} \right) \operatorname{pen}(v_\ell) \ge 0.
\end{equation}
Hence, the bound \eqref{EquicoercivityofEdelta1(2)} implies uniform control over $\int_K |\varDelta \wv_\ell|^2$ and $\operatorname{pen}(v_\ell)$.

To estimate $\| \wv_\ell \|_{H^1(\Omega)}$, we integrate by parts elementwise:
\begin{equation}\label{eq:H1_bound}
\int_\Omega | \nabla \wv_\ell |^2 = \sum_{e \in E_{h(\ell)}} \int_e \wv_\ell \jumpop{ \nabla \wv_\ell \cdot n } + \sum_{K \in T_{h(\ell)}} \int_K \wv_\ell \varDelta \wv_\ell.
\end{equation}
For the boundary terms, observe that since the boundary conditions are weakly imposed and $g=0$, we have $\| \wv_\ell \|^2_{L^2(\partial \Omega)} \lesssim h$. Using Cauchy–Schwarz and inverse estimates:
\begin{equation}
\sum_{e \in E^b_{h(\ell)}} \int_e | \wv_\ell \nabla \wv_\ell \cdot n | 
\lesssim \sum_{e \in E^b_{h(\ell)}} \| \wv_\ell \|_{L^2(e)} \| \nabla \wv_\ell \|_{L^2(e)} 
\lesssim \sum_{e \in E^b_{h(\ell)}} \frac{1}{h_e} \| \wv_\ell \|^2_{L^2(e)} 
\le \operatorname{pen}(v_\ell).
\end{equation}

For the remaining terms in \eqref{eq:H1_bound}, we use the Cauchy–Schwarz and Young inequalities with a trace–Poincaré inequality (see \cite[Lemma 10.2.20]{brenner2005c}):
\begin{equation}\label{eq:H1_internalbounds}
\begin{aligned}
&\left| \sum_{e \in E^i_{h(\ell)}} \int_e \wv_\ell \jumpop{ \nabla \wv_\ell \cdot n } + \sum_{K \in T_{h(\ell)}} \int_K \wv_\ell \varDelta \wv_\ell \right| \\
  &\lesssim \| \wv_\ell \|_{L^2(\Omega)} \left[ \left( \sum_{e \in E^i_{h(\ell)}} \frac{1}{h_e} \int_e | \jumpop{ \nabla \wv_\ell } |^2 \right)^{1/2} + \left( \sum_{K \in T_{h(\ell)}} \int_K | \varDelta \wv_\ell |^2 \right)^{1/2} \right] \\
&\lesssim \| \wv_\ell \|_{L^2(\Omega)} 
\lesssim \frac{\varepsilon}{2} \| \nabla \wv_\ell \|_{L^2(\Omega)}^2 + \frac{\varepsilon}{2} \| \wv_\ell \|^2_{L^2(\partial \Omega)} + \frac{C}{2\varepsilon},
\end{aligned}
\end{equation}
for arbitrary $\varepsilon \in (0,1)$. Since $\| \wv_\ell \|^2_{L^2(\partial \Omega)}$ is uniformly bounded (via the penalty), we conclude by choosing $\varepsilon$ small enough to absorb the $H^1$ term into the left-hand side.
\end{proof}

\subsection{Convergence of the Discrete Minimisers}

In this section, we prove that minimisers of the discrete functional $\mathcal{E}_\ell$ converge (up to interpolation) to the unique minimiser of the continuous functional $\mathcal{E}$, as the discretisation parameter $\ell \to \infty$.

The key analytical tool is $\Gamma$-convergence, specifically the standard $\liminf$–$\limsup$ framework. We show that the interpolated discrete minimisers $I_{\S^q_{h(\ell)}} u_\ell$ converge weakly in $H^1(\Omega)$ to the exact solution $u$ of the continuous variational problem, provided the approximation parameters are scaled appropriately.

\begin{theorem}[Convergence of discrete minimisers]
\label{Thrm:Gamma_funct}
Let $\mathcal{E}$ and $\mathcal{E}_\ell$ be the energy functionals defined in \eqref{energy_fun} and \eqref{E_ell}, respectively, and let $f \in C^0(\bar\Omega)$. Let $u_\ell \in V_\ell$ be a minimiser of $\mathcal{E}_\ell$, and define the interpolated function $\widehat u_\ell := I_{\S^q_{h(\ell)}} u_\ell$.

Let
\[
\beta_\ell := \max \left\{ \tilde \beta^{[1, 0, 2]}_\ell,\, \tilde \beta^{[2, 1, 2]}_\ell,\, \tilde \beta^{[3, 2, 2]}_\ell,\, \tilde \beta^{[4, 2, 2]}_\ell \right\},
\]
and suppose that the finite element parameters satisfy $h(\ell) = c \beta_\ell$ for some $c > 0$, and further that
\[
\underline h_{E, \ell}^{-1/2} \left( \tilde \beta^{[2,0,\infty]}_\ell \right)^{1 - 2\epsilon} \le C
\]
for all $\ell$, where $\underline h_{E, \ell} := \min_{e \in E^b_{h(\ell)}} h_e$.

Then, as $\ell \to \infty$,
\begin{equation}
\widehat u_\ell \rightharpoonup u \quad \text{in } H^1(\Omega), \quad \text{and} \quad
\widehat u_\ell \to u \quad \text{in } L^2(\Omega),
\end{equation}
where $u$ is the unique minimiser of the continuous energy $\mathcal{E}$, and in particular satisfies $u \in V(\Omega) = \{ v \in H_0^1(\Omega)\,:\, \Delta v \in L^2(\Omega) \}$.
\end{theorem}

The proof of Theorem~\ref{Thrm:Gamma_funct} follows the standard strategy of $\Gamma$-convergence, and is structured around the following Lemma~\ref{lem:liminf}, showing lower semicontinuity and a $\liminf$ inequality, and Lemma \ref{lem:limsup} giving a proof of the $\limsup$ inequality.

\begin{lemma}[Lower semicontinuity and $\liminf$ inequality]
\label{lem:liminf}
Let $v \in V(\Omega)$ and suppose that a sequence $(v_\ell)$ with $v_\ell \in V_\ell$ satisfies
\[
\widehat v_\ell := I_{\S^q_{h(\ell)}} v_\ell \rightharpoonup v \quad \text{in } H^1(\Omega).
\]
Then
\begin{equation} \label{liminf_Th}
\mathcal{E}(v) \leq \liminf_{\ell \rightarrow \infty} \mathcal{E}_\ell(v_\ell).
\end{equation}
\end{lemma}

\begin{proof}
Assume there exists a subsequence, still denoted by $v_\ell$, such that $\mathcal{E}_\ell(v_\ell) \leq C$ uniformly in $\ell$; otherwise the inequality holds trivially. Then the stability result of Proposition~\ref{Prop:EquicoercivityofE(2)} implies the uniform bounds of~\eqref{EquicoercivityofEdelta2(2)}.

These bounds, together with $\widehat v_\ell \rightharpoonup v$ in $H^1(\Omega)$, imply that
\[
L_{h(\ell)}(\widehat v_\ell) \rightharpoonup \Delta v, \quad \text{in } L^2(\Omega),
\]
by Lemma~\ref{Lem:dlapl_weak_conv}. Since $f_\ell \to f$ in $L^2(\Omega)$, we also have
\[
L_{h(\ell)}(\widehat v_\ell) + f_\ell \rightharpoonup \Delta v + f.
\]
By the weak lower semicontinuity of the $L^2$ norm,
\[
\liminf_{\ell \to \infty} \|L_{h(\ell)}(\widehat v_\ell) + f_\ell\|^2 \geq \int_\Omega |\Delta v + f|^2 = \mathcal{E}(v).
\]

To complete the argument, we estimate the full energy:
\begin{equation} \label{eq:liminf1stbnd}
\begin{aligned}
\liminf_{\ell \to \infty} \mathcal{E}_\ell(v_\ell)
&= \liminf_{\ell \to \infty} \left[ \|L_{h(\ell)}(\widehat v_\ell) + f_\ell\|^2 + |R_h(\nabla \widehat v_\ell)|^2 + \alpha\,\text{pen}(\widehat v_\ell) \right] \\
&\ge \liminf_{\ell \to \infty} \left[ \|L_{h(\ell)}(\widehat v_\ell) + f_\ell\|^2 + (\alpha - C_R)\,\text{pen}(\widehat v_\ell) \right] \\
&\ge \mathcal{E}(v),
\end{aligned}
\end{equation}
provided that $\alpha > C_R$.
\end{proof}

\begin{lemma}[$\limsup$ inequality via recovery sequence]
\label{lem:limsup}
For every $w \in V(\Omega)$, there exists a sequence $(w_\ell)$ with $w_\ell \in V_\ell$ such that
\[
\mathcal{E}(w) = \lim_{\ell \to \infty} \mathcal{E}_\ell(w_\ell).
\]
\end{lemma}

\begin{proof}
Let $w \in V(\Omega)$ be arbitrary. For each $\delta > 0$, choose a smooth function $w_\delta \in C^m(\overline{\Omega})$, $m \geq 4$, with $\restr{w_\delta}{\partial \Omega} = 0$, see Remark \ref{smoothing},  such that
\begin{equation}
\|w - w_\delta\|_{H^1(\Omega)} + \|\Delta w - \Delta w_\delta\|_{L^2(\Omega)} \lesssim \delta,
\quad
|w_\delta|_{H^s(\Omega)} \lesssim \delta^{1-s}.
\end{equation}
Then for each $\ell$, there exists $w_{\ell,\delta} \in V_\ell \cap H^3(\Omega)$ such that
\begin{equation} \label{nnd_approx}
\|w_{\ell,\delta} - w_\delta\|_{H^{s-1}(\Omega)} \leq \tilde\beta_\ell^{[s,s-1,2]}\, |w_\delta|_{H^s(\Omega)}
\leq \tilde\beta_\ell^{[s,s-1,2]} \delta^{1-s},
\end{equation}
and $\tilde\beta_\ell^{[s,s-1,2]} \to 0$ as $\ell \to \infty$ for all $s \le 4$.

Let $\widehat w_{\ell,\delta} := I_{\S^q_{h(\ell)}} w_{\ell,\delta}$. We aim to show
\begin{equation}
\big| \| \Delta w_{\ell,\delta} + f_\ell \|_{L^2}^2 - \mathcal{E}_\ell(w_{\ell,\delta}) \big| \to 0,
\qquad
\big| \mathcal{E}(w) - \| \Delta w_{\ell,\delta} + f_\ell \|_{L^2}^2 \big| \to 0,
\end{equation}
as $\ell \to \infty$. This proves the $\limsup$ inequality with $w_\ell := w_{\ell,\delta}$.

Using~\eqref{nnd_approx} and the approximation bounds from~\eqref{error_est}, we obtain:
\begin{equation}
  \label{eq:L2w_ld}
  \| \Delta w_{\ell,\delta} + f_\ell \|_{L^2(\Omega)} \lesssim \frac{\tilde\beta^{[3,2,2]}}{\delta^2} + \frac{1}{\delta} + 1.
\end{equation}
We then estimate the difference using $||a| -|b|| \le |a -b|$
\begin{equation}\label{dif_LL2norms}
\begin{aligned}
&\big| \| \Delta w_{\ell,\delta} + f_\ell \|_{L^2(K)}^2 
    - \| \Delta \widehat w_{\ell,\delta} + f_\ell \|_{L^2(K)}^2 \big| \\
&\quad= \big| \| \Delta w_{\ell,\delta} + f_\ell \|_{L^2(K)} 
         - \| \Delta \widehat w_{\ell,\delta} + f_\ell \|_{L^2(K)} \big| \\
&\qquad\quad \times \big( \| \Delta w_{\ell,\delta} + f_\ell \|_{L^2(K)} 
         + \| \Delta \widehat w_{\ell,\delta} + f_\ell \|_{L^2(K)} \big) \\
&\quad\lesssim \| \Delta w_{\ell,\delta} - \Delta \widehat w_{\ell,\delta} \|_{L^2(K)}  
        \big( \| \Delta w_{\ell,\delta} + f_\ell \|_{L^2(K)} 
         + \| \Delta w_{\ell,\delta} - \Delta \widehat w_{\ell,\delta} \|_{L^2(K)} \big) \\
&\quad\lesssim h(\ell) \, |w_{\ell,\delta}|_{H^3(K)} 
        \cdot \| \Delta w_{\ell,\delta} + f_\ell \|_{L^2(K)} 
        + h(\ell)^2 \, |w_{\ell,\delta}|_{H^3(K)}^2.
\end{aligned}
\end{equation}

Summing over $K \in T_h(\ell)$, and using the Cauchy–Schwarz inequality, \eqref{nnd_approx}, and \eqref{eq:L2w_ld}, we obtain the upper bound
\begin{equation}
\begin{aligned}
\bigg| \sum_{K \in T_{h(\ell)}} \Big( &\| \Delta w_{\ell, \delta} + f_\ell \|^2_{L^2(K)} 
- \| \Delta \widehat w_{\ell, \delta} + f_\ell \|^2_{L^2(K)} \Big) \bigg| 
\\
&\le \sum_{K \in T_{h(\ell)}} \big| \| \Delta w_{\ell, \delta} + f_\ell \|^2_{L^2(K)} 
- \| \Delta \widehat w_{\ell, \delta} + f_\ell \|^2_{L^2(K)} \big| 
\\
&\le \sum_{K \in T_{h(\ell)}} h(\ell) |w_{\ell, \delta}|_{H^3(K)} 
\| \Delta w_{\ell, \delta} + f_\ell \|_{L^2(K)} 
+ h^2(\ell) |w_{\ell, \delta}|_{H^3(\Omega)}^2 
\\
&\le h(\ell) |w_{\ell, \delta}|_{H^3(\Omega)} \| \Delta w_{\ell, \delta} + f_\ell \|_{L^2(\Omega)} 
+ h^2(\ell) |w_{\ell, \delta}|_{H^3(\Omega)}^2 
\\
&\lesssim \frac{h(\ell)}{\delta^3} \left[ \left( \tilde \beta_\ell^{[4,3,2]} + \delta \right) 
\left( \frac{ \tilde \beta_\ell^{[3,2,2]} }{\delta^2} + \frac{1}{\delta} + 1 \right)
+ \frac{h(\ell)}{\delta^3} \left( \left( \tilde \beta_\ell^{[4,3,2]} \right)^2 + \delta^2 \right) \right].
\end{aligned}
\label{uhd_high_order}
\end{equation}

We now fix $h(\ell) = \beta_\ell$ and choose $\delta = h(\ell)^{1/4}$. With this choice, the estimate \eqref{uhd_high_order} tends to zero as $\ell \to \infty$.

It remains to show that the remaining terms of $\mathcal{E}_\ell(w_{\ell,\delta})$, namely the consistency and penalty terms in \eqref{E_h}, also vanish as $\ell \to \infty$.

\textit{Consistency term.} Using \eqref{eq:cons_bnd1} and \eqref{nnd_approx}, we estimate:
\begin{equation}
\begin{aligned}
\left| \sum_{e \in E_h^i} \int_e \dgal{ \Delta \widehat w_{\ell,\delta} + f_\ell } \, \jumpop{ \nabla \widehat w_{\ell,\delta} \cdot n } \right|
&\lesssim h(\ell) \left( |w_{\ell,\delta}|_{H^2(\Omega)} + 1 \right) |w_{\ell,\delta}|_{H^3(\Omega)} \\
&\lesssim h(\ell) \left( |w_{\ell,\delta} - w_\delta|_{H^2(\Omega)} + |w_\delta|_{H^2(\Omega)} + 1 \right) \\
&\qquad \times \left( |w_{\ell,\delta} - w_\delta|_{H^3(\Omega)} + |w_\delta|_{H^3(\Omega)} \right) \\
&\lesssim h(\ell) \left( \frac{\tilde \beta_\ell^{[3,2,2]}}{\delta^2} + \frac{1}{\delta} + 1 \right)
\left( \frac{\tilde \beta_\ell^{[4,3,2]}}{\delta^3} + \frac{1}{\delta^2} \right) \\
&\to 0 \quad \text{as } \ell \to \infty.
\end{aligned}
\end{equation}

\textit{Penalty term.} Using \eqref{eq:jumps_bnd1} and \eqref{nnd_approx}, we estimate:
\begin{equation}\label{suppen2}
\begin{aligned}
\sum_{e \in E^i_h} \frac{1}{h_e} \int_e \left| \jumpop{ \nabla \widehat w_{\ell,\delta} } \right|^2 
&\lesssim h(\ell)^2 \, |w_{\ell,\delta}|_{H^3(\Omega)}^2 \\
&\lesssim \frac{h(\ell)^2}{\delta^4} \left( \frac{ \left( \tilde \beta_\ell^{[4,3,2]} \right)^2 }{\delta^2} + 1 \right) \\
&= h(\ell) + h(\ell)^{3/4} \to 0 \quad \text{as } \ell \to \infty.
\end{aligned}
\end{equation}

\textit{Boundary penalty terms.} Since $w_\delta = 0$ on $\partial \Omega$, we use \eqref{ineq:error_on_e} and \eqref{nnd_approx} to estimate:
\begin{equation}\label{suppen4}
\begin{aligned}
\sum_{e \in E^b_{h(\ell)}} \frac{1}{h_e} \| \widehat w_{\ell,\delta} \|_{L^2(e)}^2 
&\lesssim \sum_{e \in E^b_{h(\ell)}} \frac{1}{h_e} \left( \| \widehat w_{\ell,\delta} - w_{\ell,\delta} \|_{L^2(e)}^2 
+ \| w_{\ell,\delta} - w_\delta \|_{L^2(e)}^2 \right) \\
&\lesssim h(\ell)^2 |w_{\ell,\delta}|_{H^2(\Omega)}^2 
+ \frac{1}{h(\ell)} \| w_{\ell,\delta} - w_\delta \|_{L^2(\Omega)} 
\cdot | w_{\ell,\delta} - w_\delta |_{H^1(\Omega)}.
\end{aligned}
\end{equation}

The first term vanishes as $\ell \to \infty$:
\begin{equation}
\begin{aligned}
h(\ell)^2 |w_{\ell,\delta}|_{H^2(\Omega)}^2 
&\lesssim h(\ell)^2 \left( |w_{\ell,\delta} - w_\delta|_{H^2(\Omega)}^2 + |w_\delta|_{H^2(\Omega)}^2 \right) \\
&\lesssim h(\ell)^2 \left( \left( \tilde \beta_\ell^{[3,2,2]} \right)^2 |w_\delta|_{H^3(\Omega)}^2 
+ \frac{1}{\delta^2} \right) \\
&\lesssim h(\ell)^2 \left( \frac{ \left( \tilde \beta_\ell^{[3,2,2]} \right)^2 }{\delta^4} + \frac{1}{\delta^2} \right).
\end{aligned}
\end{equation}

For the second term in \eqref{suppen4}, we estimate:
\begin{equation}
\begin{aligned}
\frac{1}{h(\ell)} \| w_{\ell,\delta} - w_\delta \|_{L^2(\Omega)} 
\cdot | w_{\ell,\delta} - w_\delta |_{H^1(\Omega)} 
&\lesssim \frac{ \tilde \beta_\ell^{[1,0,2]} \tilde \beta_\ell^{[2,1,2]} }{h(\ell)} 
\| w_\delta \|_{H^1(\Omega)} \| w_\delta \|_{H^2(\Omega)} \\
&\lesssim \frac{ \beta_\ell^2 }{ \delta h(\ell) } (\delta + 1) 
= h(\ell)^{3/4} ( h(\ell)^{1/4} + 1 ) \to 0,
\end{aligned}
\end{equation}
completing the proof.
\end{proof}

\begin{proof}[Proof of Theorem~\ref{Thrm:Gamma_funct}]
  Since $u_\ell$ is a minimiser of $\mathcal{E}_\ell$, we have for any $v_\ell \in V_\ell$,
\[
\mathcal{E}_\ell(u_\ell) \le \mathcal{E}_\ell(v_\ell).
\]
In particular, taking $v_\ell = \tilde u_\ell$, the recovery sequence associated with $w = u$, we obtain
\[
\mathcal{E}_\ell(u_\ell) \le \mathcal{E}_\ell(\tilde u_\ell).
\]
By construction, $\mathcal{E}_\ell(\tilde u_\ell) \to \mathcal{E}(u)$, and hence the sequence $(\mathcal{E}_\ell(u_\ell))$ is bounded. The stability result of Proposition~\ref{Prop:EquicoercivityofE(2)} then implies the uniform bound
\begin{equation}\label{conv_bounds}
\sum_{K \in T_{h(\ell)}} \| \Delta \widehat u_\ell \|^2_{L^2(K)} + \| \widehat u_\ell \|^2_{H^1(\Omega)} + \rj_\ell(\tilde u_\ell) + \rb_\ell(\tilde u_\ell) \le C.
\end{equation}
By reflexivity and the Banach–Alaoglu theorem, there exists $\tilde u \in H^1(\Omega)$ such that (up to a subsequence)
\[
\widehat u_\ell \rightharpoonup \tilde u \quad \text{in } H^1(\Omega).
\]
Moreover, Lemma~\ref{Lem:dlapl_weak_conv} together with \eqref{eq:dLaplace_bound} implies that
\[
L_{h(\ell)}(\widehat u_\ell) \rightharpoonup \Delta \tilde u \quad \text{in } L^2(\Omega),
\]
so $\Delta \tilde u \in L^2(\Omega)$. It remains to show that $\tilde u = u$.

From the bound on the penalty term $\operatorname{pen}(u_\ell)$, we have
\begin{align*}
\| \tilde u \|_{L^2(\partial \Omega)} 
&\le \| \tilde u - \widehat u_\ell \|_{L^2(\partial \Omega)} + \| \widehat u_\ell \|_{L^2(\partial \Omega)} \\
&\lesssim \| \tilde u - \widehat u_\ell \|_{L^2(\Omega)} \| \tilde u - \widehat u_\ell \|_{H^1(\Omega)} + h(\ell) \to 0,
\end{align*}
as $\ell \to \infty$, using that weak convergence in $H^1(\Omega)$ implies strong convergence in $L^2(\Omega)$, cf.~\cite[Thm.~9.16]{brezis2010functional}. Thus, $\tilde u \in V(\Omega)$.

Now let $w \in V(\Omega)$ be arbitrary, and $w_\ell \in V_\ell$ its recovery sequence. Then
\[
\mathcal{E}(\tilde u) \le \liminf_{\ell \to \infty} \mathcal{E}_\ell(u_\ell) \le \limsup_{\ell \to \infty} \mathcal{E}_\ell(u_\ell) \le \limsup_{\ell \to \infty} \mathcal{E}_\ell(w_\ell) = \mathcal{E}(w).
\]
Since $w$ was arbitrary, we conclude that $\tilde u$ is a minimiser of $\mathcal{E}$, and hence $\tilde u = u$ by uniqueness. As all subsequences converge to $u$, the full sequence satisfies
\[
\widehat u_\ell \rightharpoonup u \quad \text{in } H^1(\Omega), \quad \text{and} \quad
\widehat u_\ell \to u \quad \text{in } L^2(\Omega).
\]
\end{proof}

\begin{remark}[Approximation by smooth functions] \label {smoothing}
Let $w \in V(\Omega)$ be arbitrary. By the elliptic regularity estimates we know that $w\in H^2 (\Omega). $ For each $\delta > 0$, standard density results imply the existence of a smooth function  $w_\delta \in C^m(\overline{\Omega})$, $m \geq 4$,  such that
\begin{equation}
\|w - w_\delta\|_{H^1(\Omega)} + \|\Delta w - \Delta w_\delta\|_{L^2(\Omega)} \lesssim \delta.
\end{equation}	
To see why this function can be chosen so that   $\restr{w_\delta}{\partial \Omega} = 0$, one has to follow the proof of 
Theorem 3 in \cite[Section 5.3]{Evans2} and use as mollifiers on the boundary convolutions with odd extension $\bar w (x) = w(x)$ if $x\in B^+$ and $\bar w (x) = - w(x_1, \dots , -x_d )$ if $x\in B^-$ where we follow the notation of  \cite[ Theorem 1, Section 5.4]{Evans2}. Then using the fact that $w=0$ on the boundary the arguments of  \cite[ Theorem 1, Section 5.4]{Evans2} show that 
$\bar w $ is smooth enough, odd with respect to $x_d$ and zero on the local boundary $x_d=0. $ It then follows that the corresponding mollifier is zero on the boundary. The property $|w_\delta|_{H^s(\Omega)} \lesssim \delta^{1-s}$ follows by the fact that 
the approximation constructed is a finite sum of convolutions. 
\end{remark}

\section{Computational Experiments}\label{CE}

We now evaluate the numerical performance and robustness of the proposed method. In particular, we demonstrate its ability to approximate minimisers of the continuous energy functional \eqref{energy_fun} with high accuracy, and highlight its robustness under mesh refinement and varying neural network architectures. Additionally, we show how the proposed finite element training approach significantly reduces both the computational cost and memory usage associated with gradient evaluations of neural network functions.

The finite element discretisations used in training are based on uniform triangulations. For instance, the unit square is divided into 9 equal squares, each of which is further subdivided into two triangles, yielding a total of 18 triangles as illustrated in Figure~\ref{fig:mesh}.

\begin{figure}[h!]
\centering
\includegraphics[width=8cm]{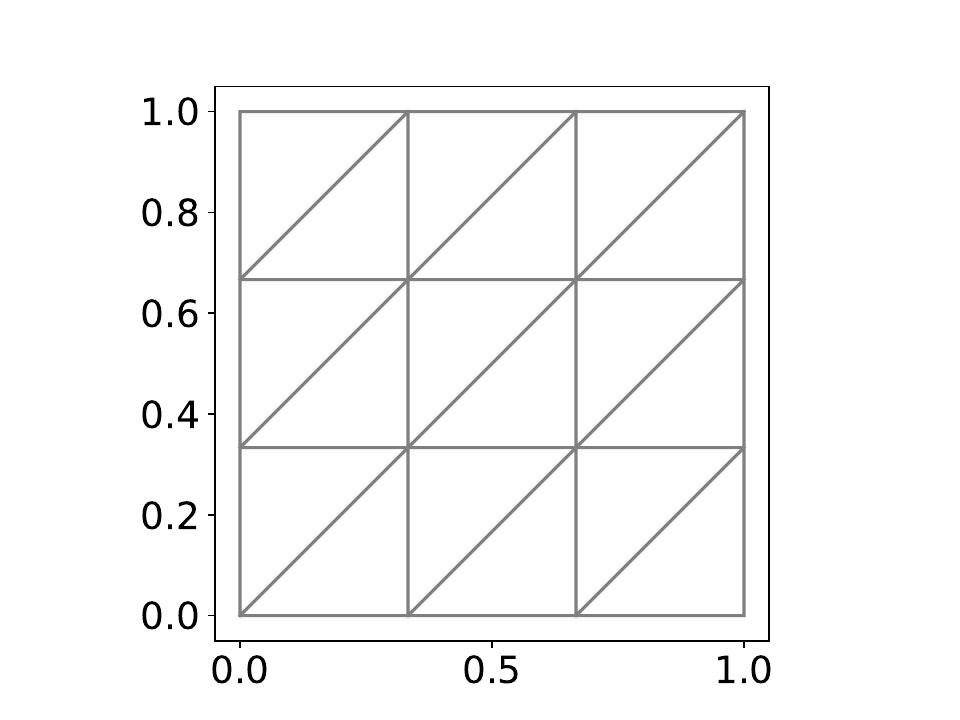}
\caption{Uniform mesh: The unit square is divided into $3 \times 3 = 9$ squares, each further subdivided into 2 triangles.}
\label{fig:mesh}
\end{figure}

Throughout this section, we use continuous piecewise polynomial finite element spaces of degree 2. The penalty parameter in the discrete energy functional \eqref{E_h} is fixed at $\alpha = 60$.

Recalling the residual neural network (ResNet) architecture defined in Equations~\eqref{C_L}, \eqref{C_k}, and \eqref{eq:blocks}, we restrict our experiments to two-dimensional domains $\Omega \subset \mathbb{R}^2$. The width of each hidden layer is fixed at $N = 64$, and we compare the approximation performance of networks with 2 and 4 residual blocks. The activation function used is $\tanh(\cdot)$, and optimisation is performed using the Adam algorithm.

All neural network components are implemented in the PyTorch framework \cite{paszke2019pytorch}, with parameters stored in \texttt{float32} precision. All computations were performed on an RTX A4500 GPU.

\subsection{Approximating Solutions in Convex Domains}

We begin by considering approximations on convex domains. Specifically, we choose the right-hand side $f$ in \eqref{SSe:1.1} such that the exact solution to Poisson's equation is given by
\begin{align}\label{eq:sol1}
u(x_1, x_2) = \sin(4\pi x_1)\sin(4\pi x_2), \qquad (x_1, x_2) \in \Omega = [0,1]^2.
\end{align}

To approximate minimisers of the discrete energy functional \eqref{E_h}, we employ a residual neural network with two blocks. We increase the resolution by refining the mesh (i.e., increasing the number of triangles) and simultaneously vary the precision of the integration rule used during training.

After $10^5$ training epochs using the Adam optimiser with learning rate $5 \times 10^{-3}$, the $L^2(\Omega)$ errors (solid curves) are reported in Figure~\ref{fig:blocks2}. The minimal observed error is approximately $9.3 \cdot 10^{-4}$, and the error generally decreases monotonically as the number of training points increases. One exception occurs for the integration rule with first-order precision, where the error increases from approximately $10^{-3}$ to $2 \cdot 10^{-3}$ (solid green curve).

Since $u_\theta \in C^2(\Omega)$, it is natural, both intuitively and from standard error estimates, to expect that the gradient jump terms 
$\jumpop{\nabla I_{\S_h} u_\theta}$ on each edge $e$ vanish as the discretisation parameter $h \to 0$. 
To explore the role of these terms, we repeat the preceding experiment while omitting the consistency term in \eqref{E_h}, i.e., we retain only the boundary contribution in the penalty term.

Figure~\ref{fig:blocks2} supports this expectation, see the dashed curves. Even with a small number of mesh cells, the error approaches its minimum value. 
This raises a natural question: what benefit, if any, is gained by including the gradient jump terms?

To better understand their role, we consider a test function designed to exhibit jump-like behaviour. Define
\begin{equation}\label{eq:1dstep}
u(x)= 
\begin{cases}
\sin(2\pi x), & \text{for } x \in \Omega_1 = [0, 0.5 - \epsilon],\\
\sin(2\pi x) - 1, & \text{for } x \in \Omega_2 = [0.5 + \epsilon, 1],
\end{cases}
\end{equation}
with a smooth transition in the intermediate region $\Omega_3 = [0.5 - \epsilon, 0.5 + \epsilon]$. Choosing $f = 4\pi^2 \sin(2\pi x)$ ensures that
$\| \varDelta u + f \|_{L^2(\Omega_1)} = \| \varDelta u + f \|_{L^2(\Omega_2)} = 0$. 
From standard approximation theory for neural networks, we know that there exists $\hat{u}_\theta \in V_\mathcal{N}$ such that
\[
\| \varDelta \hat{u}_\theta + f \|_{L^2(\Omega_1)} + \| \varDelta \hat{u}_\theta + f \|_{L^2(\Omega_2)} < \delta
\quad \text{for any } \delta > 0.
\]
This is also evident in practice, neural networks can approximate such step-like functions with high fidelity.

However, if the numerical integration rule does not place points in $\Omega_3$, the transition (or "jump") between $\Omega_1$ and $\Omega_2$ may go undetected. 
This allows discontinuous or piecewise solutions resembling \eqref{eq:1dstep} to persist undetected by the residual term alone.

Motivated by this, we examine the performance of deeper networks capable of representing such non-smooth features. Repeating the earlier experiment using a residual network 
with four blocks, we observe a marked degradation in $L^2$ error when the jump terms are omitted: the error increases by one to two orders of magnitude, compare the 
dashed curves in Figures~\ref{fig:blocks4} and \ref{fig:blocks2}. In particular, when using a second-order integration rule, the error grows from $\sim 10^{-3}$ to 
$\sim 10^{-2}$ as the number of training points increases (orange curve in Figure~\ref{fig:blocks4}), despite the associated loss decreasing monotonically (see Table~\ref{table:loss}).

By contrast, when minimising the full energy $\mathcal{E}_h$ (including jump terms), the $L^2$ errors are comparable to those obtained with the smaller architecture, again compare 
Figures~\ref{fig:blocks4} and \ref{fig:blocks2}. This illustrates that the proposed method offers improved robustness in that accurate solutions are obtained independently of the 
network depth or quadrature precision.

Standard neural network methods for solving PDEs are typically based on collocation training, see \cite[Section~1.4]{GrekasM_2025}. 
In this section, we compare the proposed finite element-based approach with standard Physics-Informed Neural Networks (PINNs), 
evaluating their accuracy, computational cost, and memory usage.

Following \cite{GrekasM_2025}, we focus on deterministic training via quadrature collocation, which has been shown to outperform 
Monte Carlo-based collocation in accuracy. Boundary conditions in the PINNs setup are imposed using Nitsche’s method, with the interpolant 
enforced only on the boundary. This hybrid treatment provides higher boundary accuracy than Monte Carlo sampling.

Using a residual network with two blocks, the $L^2$ error for quadrature collocation is of the order $10^{-4}$, comparable to the error 
obtained when omitting the gradient jump terms in our method. However, increasing the network depth to four blocks leads to error 
levels around $10^{-3}$ to $10^{-2}$, mirroring the instability observed when neglecting jump terms in Section~\ref{CE}.

Significant differences arise in computational cost\footnote{For quadrature collocation, the Laplacian is evaluated using forward-mode 
automatic differentiation, implemented via the \texttt{torch.func.hessian} function with vectorised mapping. This is more efficient in 
memory and runtime than reverse-mode differentiation.} and resource consumption. Figures~\ref{fig:timeb2} and~\ref{fig:timeb4} show the 
runtime ratios between quadrature collocation and the finite element method. For first-order quadrature rules, our method is 3-4 times faster; 
for second-order rules, the speedup increases to 5-11 times. In both network configurations (2 and 4 blocks), the advantage grows with the number 
of training points.

Memory usage follows a similar trend, as shown in Figures~\ref{fig:mbused1} and~\ref{fig:mbused2}. The finite element method exhibits 
significantly lower memory demands and slower memory growth, even for large training sets. In contrast, PINNs relying on automatic differentiation 
require substantially more memory, making them less practical for fine discretisations or deeper network architectures.

\begin{figure}[h!]
\includegraphics[width=\textwidth]{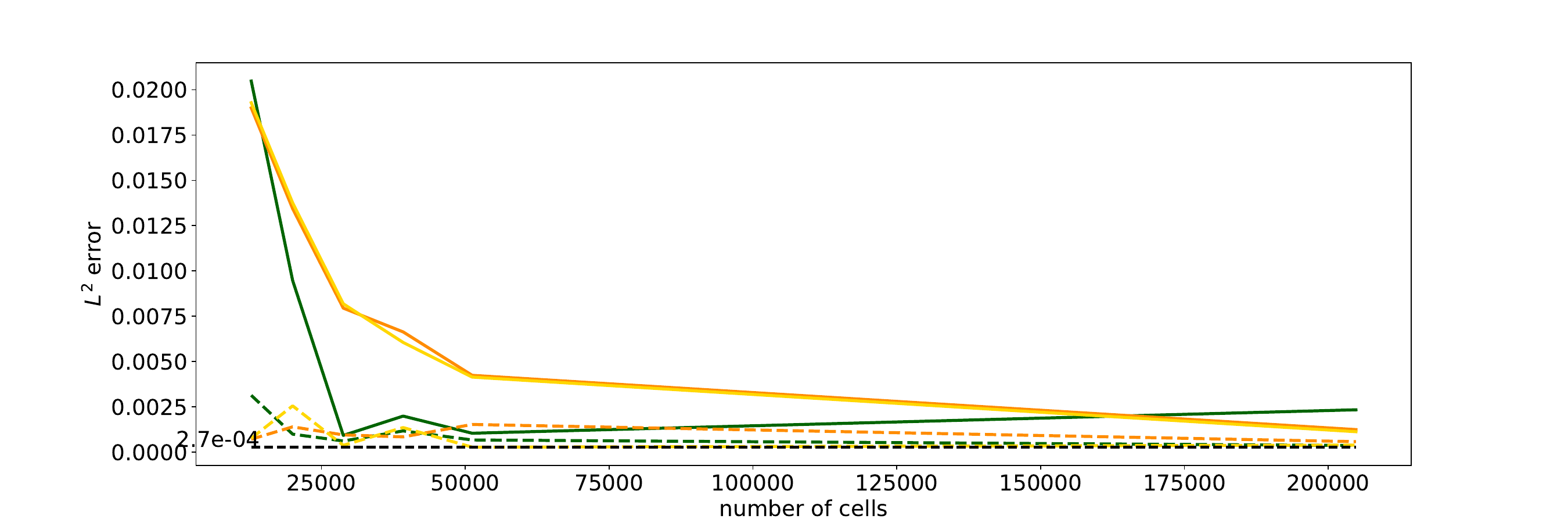}
\caption{Comparison of $L^2$-errors for solutions $u_\theta$ obtained using a residual neural network with two blocks. 
We minimise the full energy functional \eqref{E_h} (solid lines) and a reduced variant without gradient jump terms (dashed lines). 
The error $||u_\theta - u||_{L^2(\Omega)}$ is plotted against the number of mesh cells, where $u$ is the exact solution \eqref{eq:sol1}. 
Green, orange, and yellow curves correspond to quadrature rules of precision 1, 2, and 3, respectively.}
\label{fig:blocks2}
\end{figure}

\begin{figure}[h!]
\includegraphics[width=\textwidth]{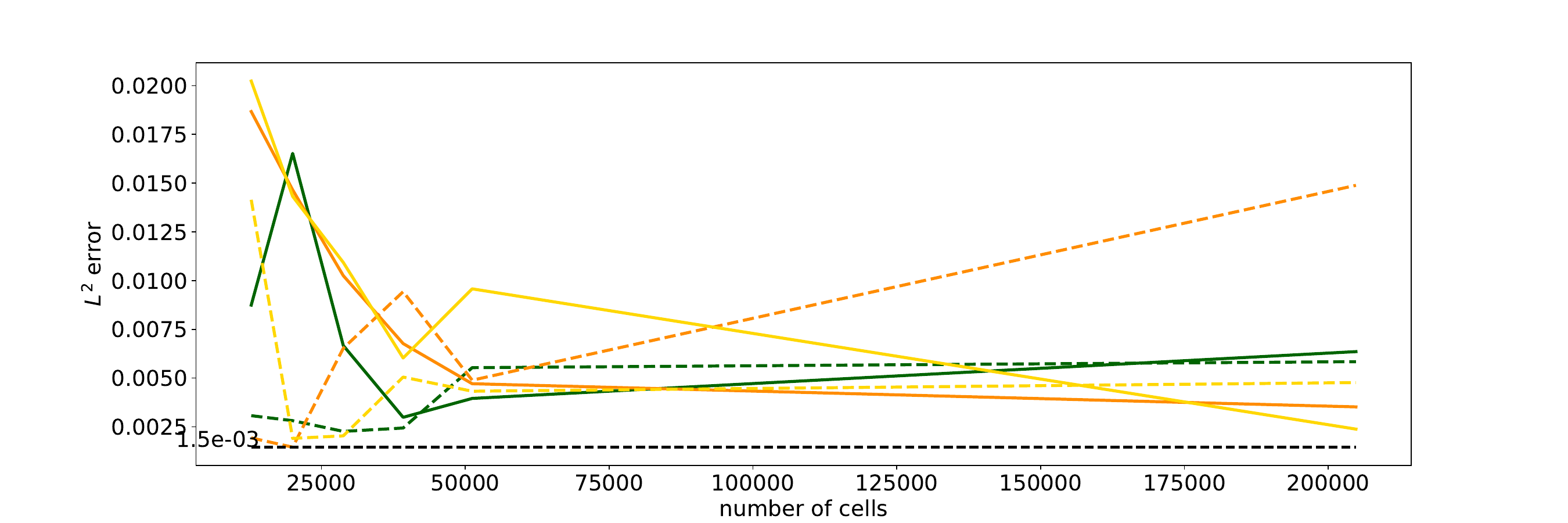}
\caption{As in Figure~\ref{fig:blocks2}, but using a residual neural network with four blocks. 
The $L^2$-error $||u_\theta - u||_{L^2(\Omega)}$ is shown for both the full energy \eqref{E_h} (solid lines) and the variant without gradient jump terms (dashed lines). 
The same integration rule precisions (1, 2, and 3) are indicated by the green, orange, and yellow curves, respectively.}
\label{fig:blocks4}
\end{figure}

\begin{table}[h!]
\centering
\begin{tabular}{||c c c c c c c||} 
 \hline
 Number of cells: & 12800 & 20000 & 28800 & 39200 & 51200 & 204800 \\ [0.5ex] 
 \hline
 Best loss:  & 76.9 & 49.4 & 34.3 & 25.4 & 19.5 & 5.1 \\ [1ex] 
 \hline
\end{tabular}
\caption{Minimum value of the discrete loss during training when gradient jump terms are omitted from \eqref{E_h}, using a quadrature rule of precision 2. As the number of mesh cells increases, the loss decreases steadily.}
\label{table:loss}
\end{table}

\begin{figure}[h!]
\includegraphics[width=\textwidth]{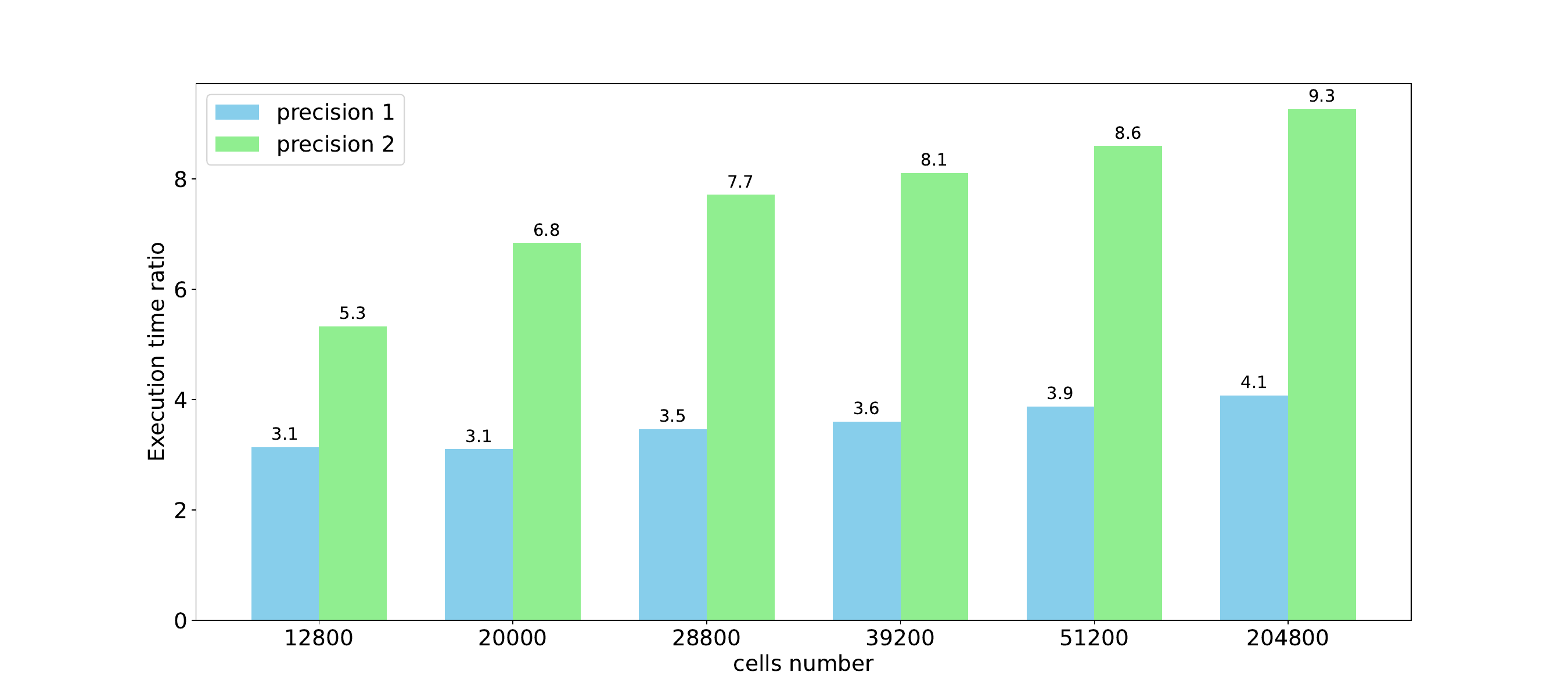}
\caption{Execution time ratio over 1,000 training iterations comparing the proposed finite element method against quadrature-based collocation, 
using a residual neural network with 2 blocks. As the number of mesh cells and integration precision increase, the finite element method becomes faster, 
achieving speedups ranging from 3.1× to 9.3×.}
\label{fig:timeb2}
\end{figure}

\begin{figure}
\includegraphics[width=\textwidth]{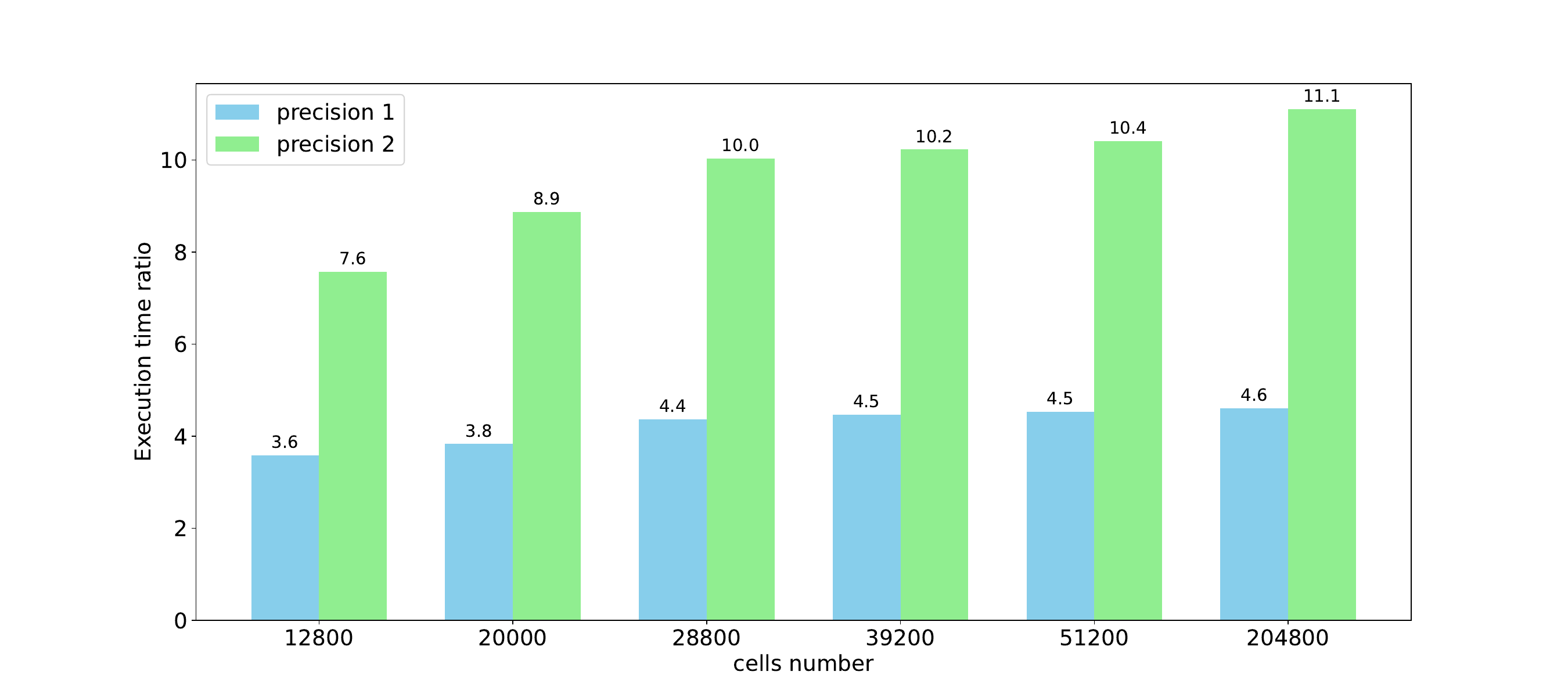}
\caption{As in Figure~\ref{fig:timeb2}, but using a residual neural network with 4 blocks. The proposed finite element method achieves speedups ranging from 3.6× to 11.1× compared to the collocation quadrature approach.}
\label{fig:timeb4}
\end{figure}

\begin{figure}[h!]
\includegraphics[width=\textwidth]{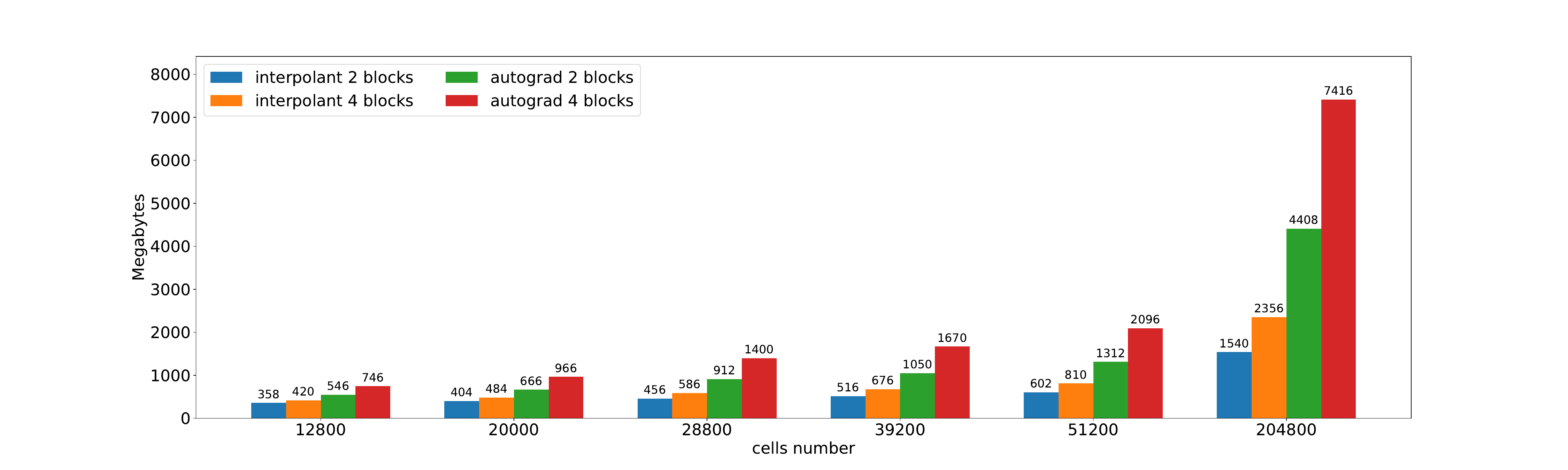}
\caption{Memory usage during training with numerical integration of precision 1, comparing the proposed method using finite element interpolation 
(interpolant) against standard automatic differentiation (autograd). Results are shown for residual neural networks with 2 and 4 blocks.}
\label{fig:mbused1}
\end{figure}

\begin{figure}[h!]
\includegraphics[width=\textwidth]{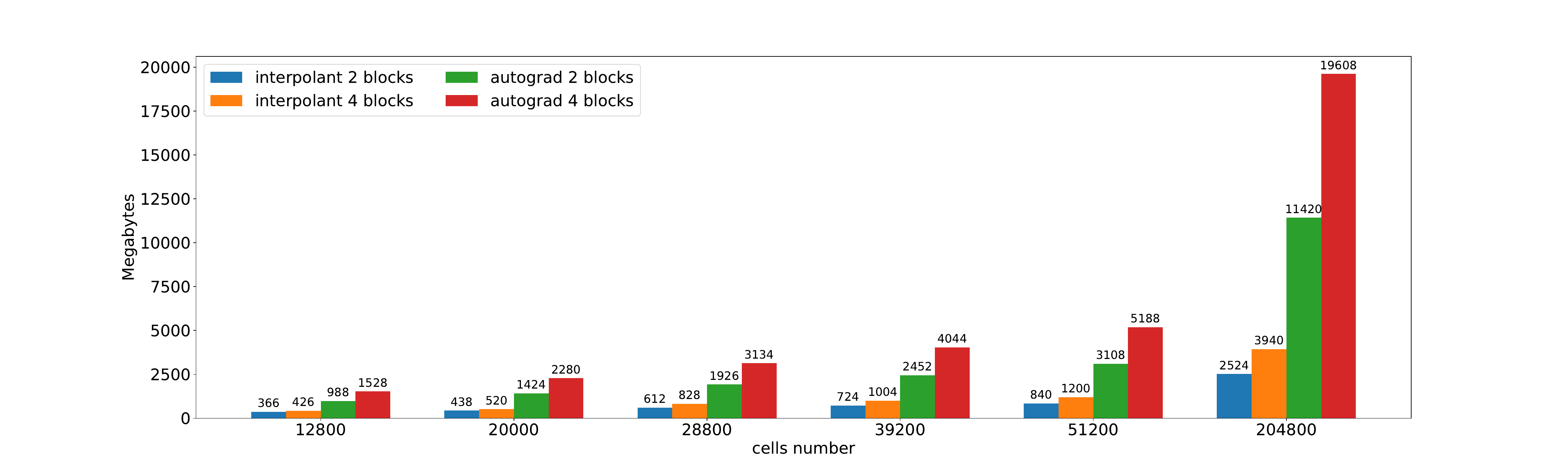}
\caption{Memory usage during training with numerical integration of precision 2, comparing the proposed method using finite element interpolation 
(interpolant) against standard automatic differentiation (autograd). Results are shown for residual neural networks with 2 and 4 blocks.}
\label{fig:mbused2}
\end{figure}

\clearpage

\subsection{Approximating solutions in nonconvex domains}

In this computational experiment, we demonstrate the robustness of training using finite elements in nonconvex domains exhibiting boundary singularities. 
Specifically, we consider the classical L-shaped domain,
\[
\Omega = (-1,1)^2 \setminus \left([0,1) \times (-1, 0]\right).
\]
In Poisson's equation~\eqref{ivp-AC}, we select the forcing term $f$ such that the exact solution is given by
\begin{equation}\label{eq:LshapedSol}
u(x,y) = (x^2 + y^2)^{\frac{2}{3}}.
\end{equation}
This solution satisfies $u \in H^1(\Omega)$ but exhibits a singularity in the Laplacian $\varDelta u$ at the reentrant corner $(0,0)$.

Using a uniform triangulation of the domain as illustrated in Figure~\ref{fig:LshapedMesh}, we approximate $u$ via the proposed scheme~\eqref{E_h}. 
Figure~\ref{fig:abs_diff} displays the pointwise absolute error $|u - u_\theta|$ for increasing mesh resolution, using a residual network with 2 blocks 
and a first-order quadrature rule. As expected, the error is concentrated near the singularity at the origin. 
Nevertheless, the $L^2$-error decreases steadily as the number of triangles increases, dropping from approximately $6 \cdot 10^{-2}$ 
with 364 cells to $9 \cdot 10^{-3}$ with 98,304 cells; see Table~\ref{table:LshapeError}.

\begin{figure}[h!]
\includegraphics[width=7cm]{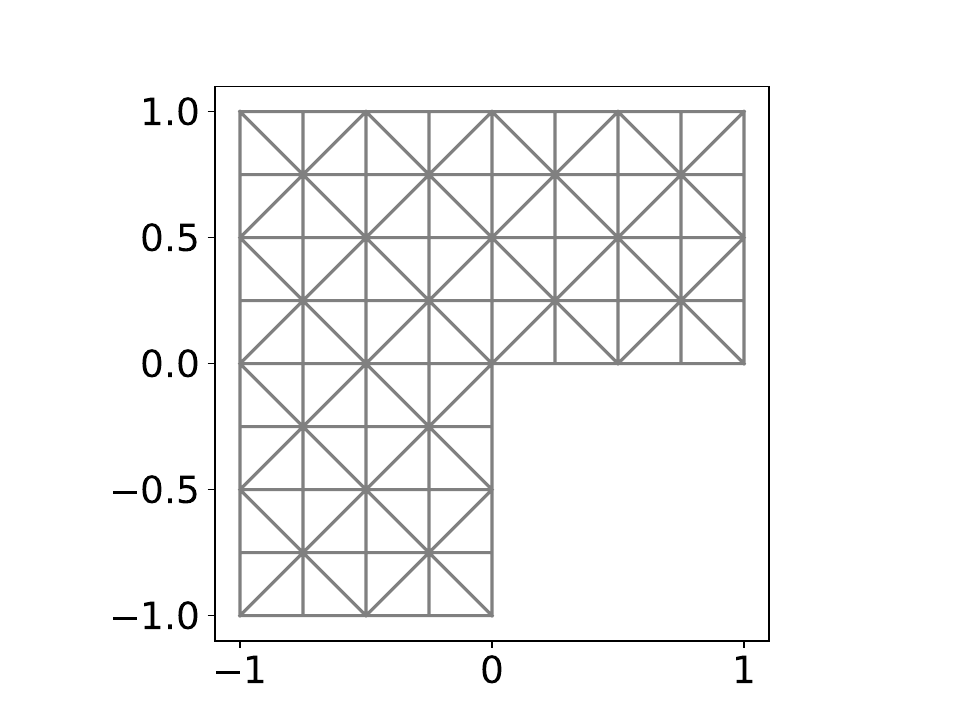}
\caption{Triangulation of the L-shaped domain. The domain is subdivided into 12 rectangles, each of which is further divided into 8 triangles.}
\label{fig:LshapedMesh}
\end{figure}

\begin{table}[h!]
\centering
\begin{tabular}{||c c c c c c||} 
 \hline
 Discretisation parameter $h$: & $\frac{1}{4}$ & $\frac{1}{8}$  & $\frac{1}{16}$ & $\frac{1}{32}$ & $\frac{1}{64}$ \\ [0.5ex] 
 \hline
 $|| u - u_\theta||_{L^2(\Omega)}$: & 0.063 & 0.036 & 0.021 & 0.020 & 0.0091 \\ [1ex] 
 \hline
\end{tabular}
\caption{Relation between the discretisation parameter $h$ and the $L^2(\Omega)$-error for the L-shaped domain example. 
Here, $u$ is the exact solution given in \eqref{eq:LshapedSol}, and $u_\theta$ is the computed minimiser.}
\label{table:LshapeError}
\end{table}

\begin{figure}
\centering
\includegraphics[width=0.45\linewidth]{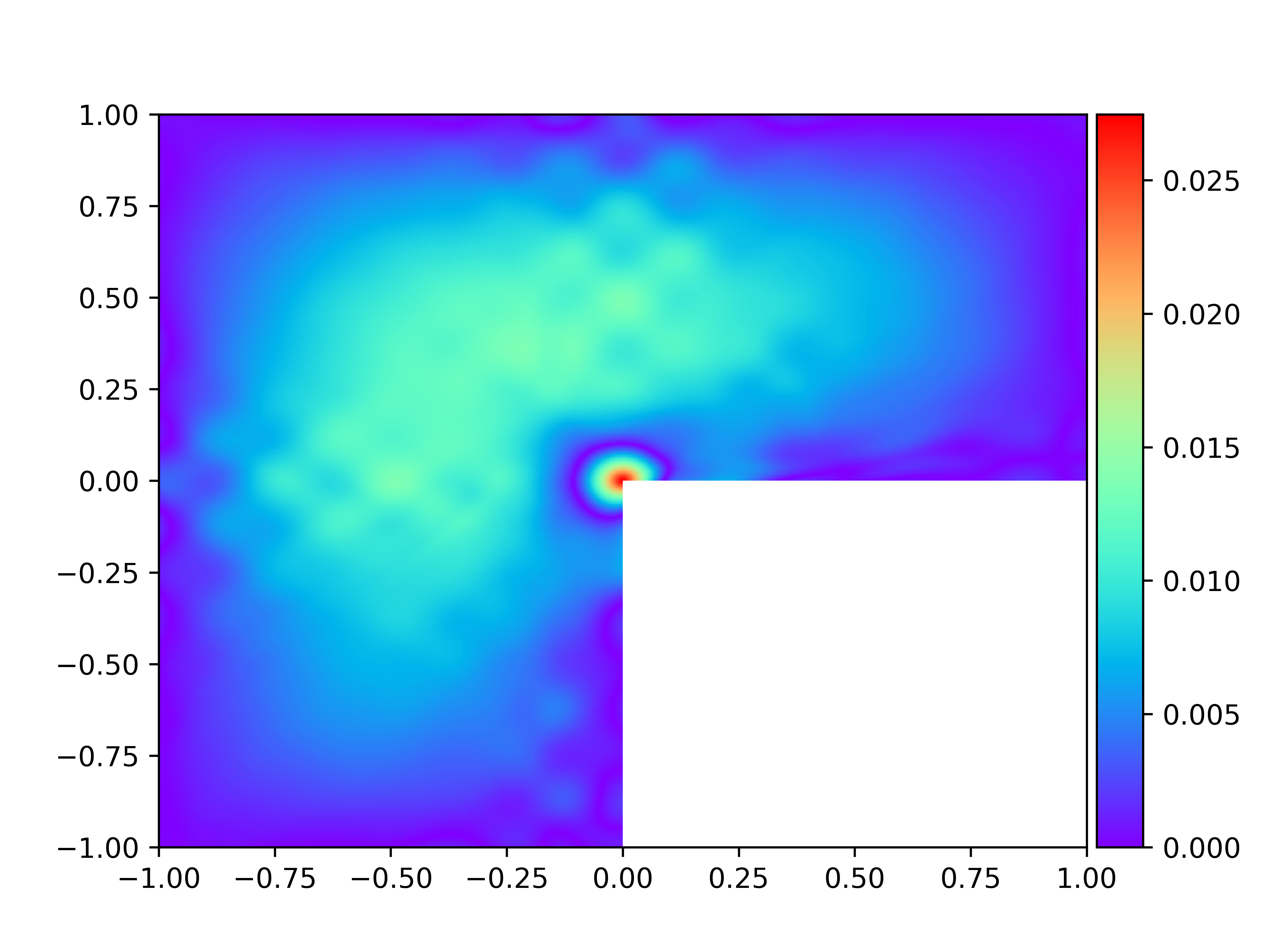}
\includegraphics[width=0.45\linewidth]{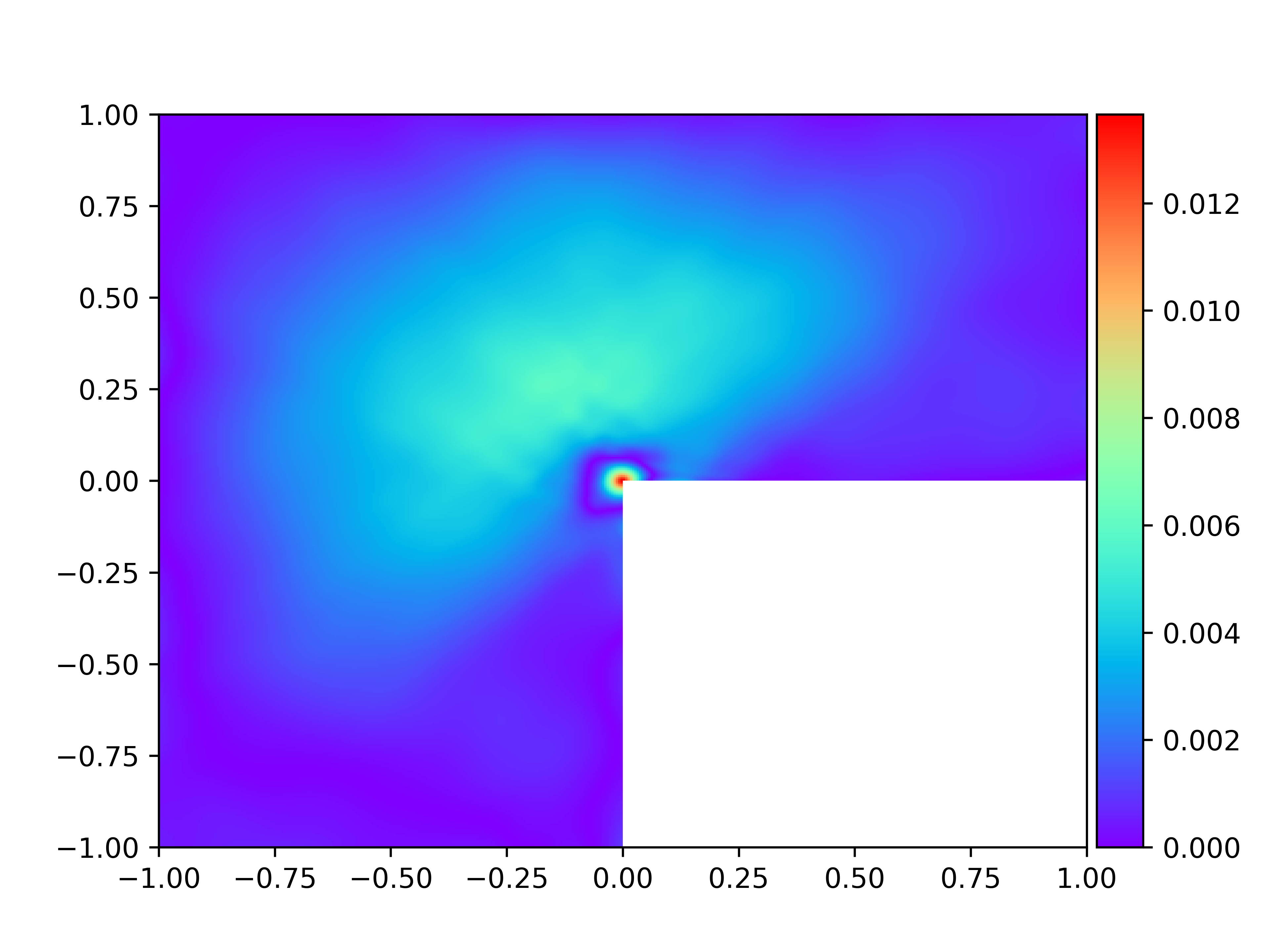}
\includegraphics[width=0.45\linewidth]{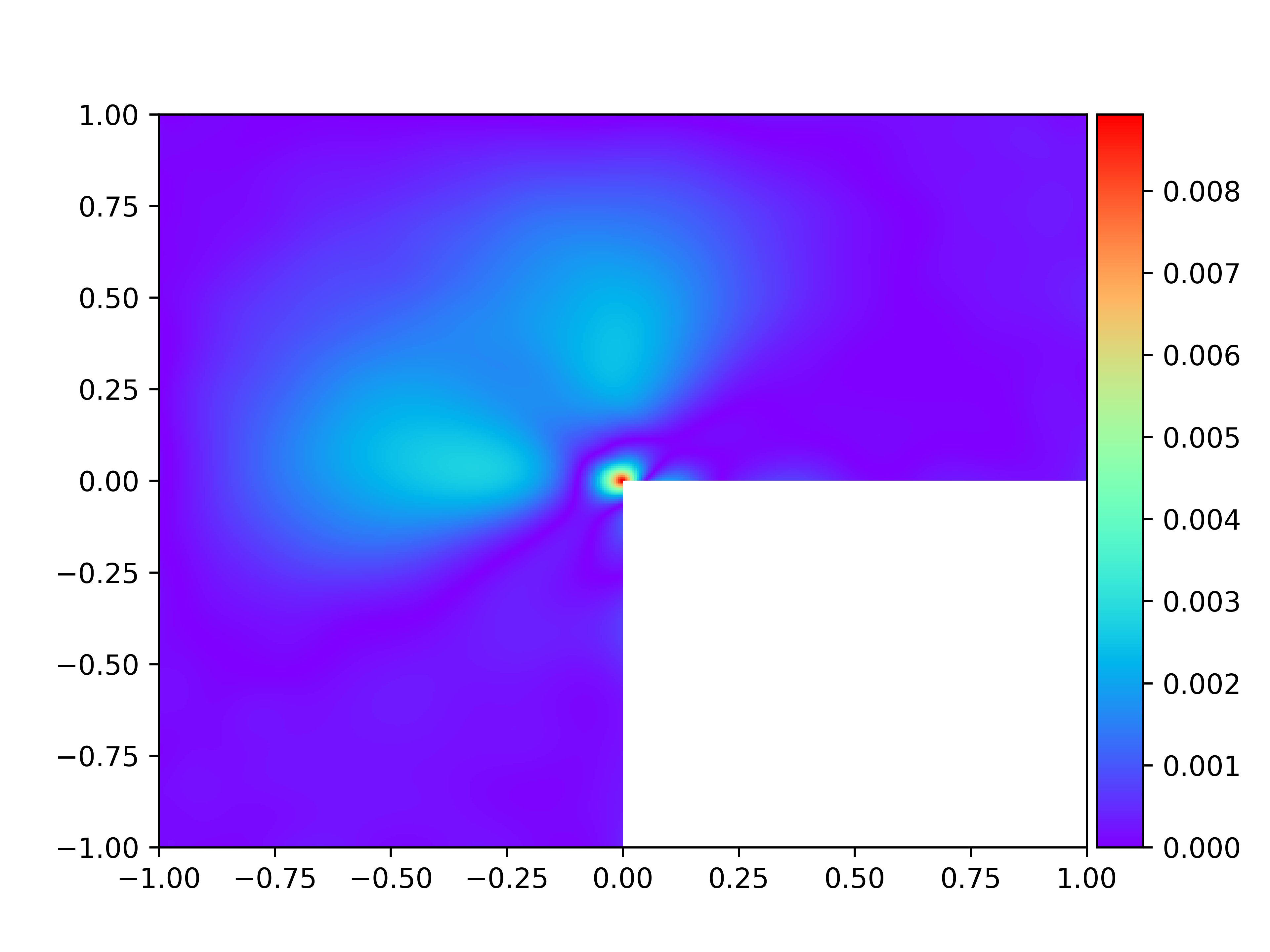}
\includegraphics[width=0.45\linewidth]{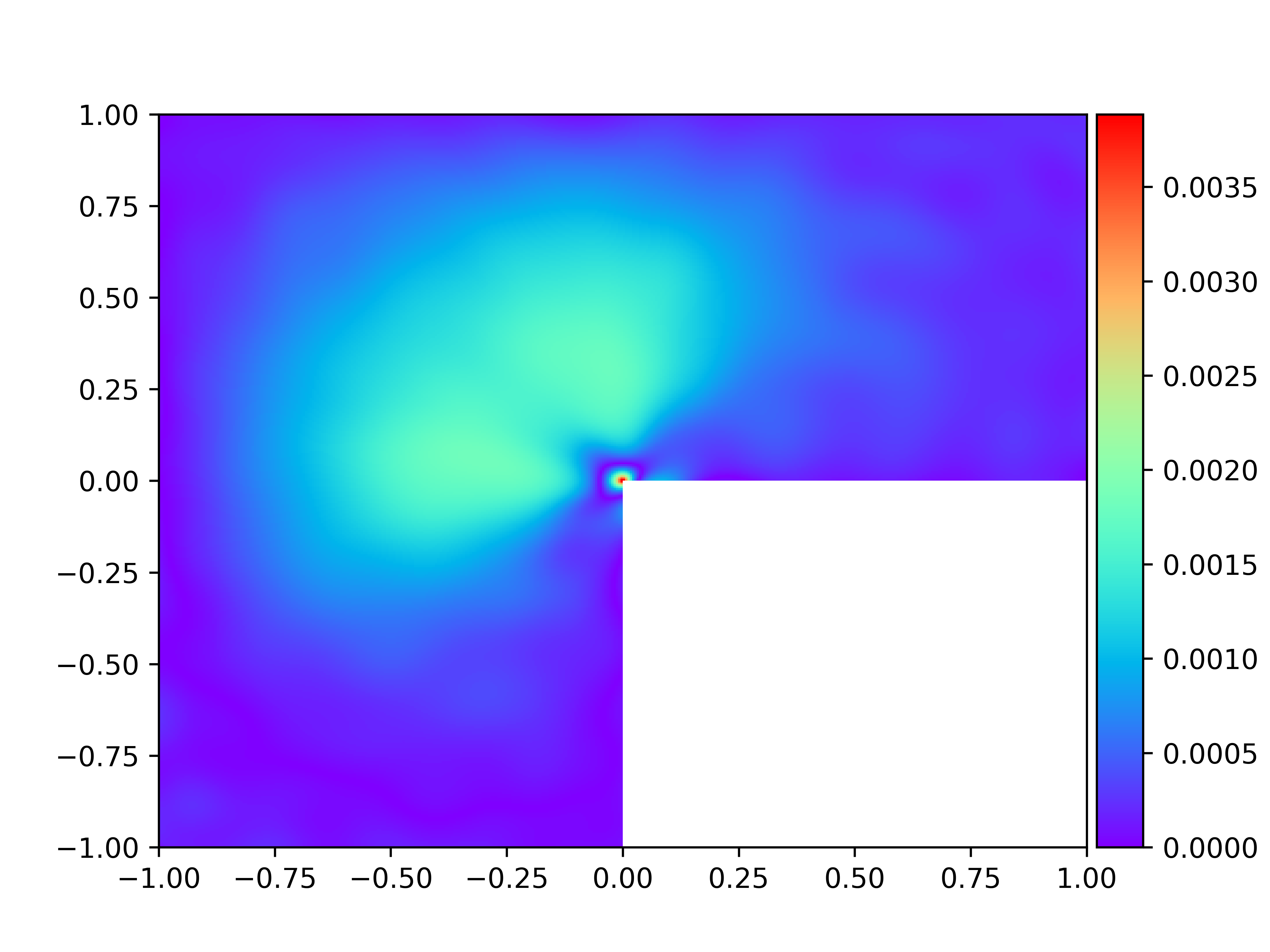}
\caption{Pointwise absolute error $|u_\theta - u|$ for varying numbers of triangles. Top row: 384 triangles ($h=1/4$), 1536 triangles ($h=1/8$); bottom row: 6144 triangles ($h=1/16$), 24576 triangles ($h=1/32$). A residual network with two blocks and numerical integration of first-order precision is used.}
\label{fig:abs_diff}
\end{figure}

\section{Conclusion}

We have introduced and analysed a finite element-based training approach for physics-informed neural networks, with a focus on energy minimisation problems. The proposed method achieves high accuracy, improved computational efficiency, and robustness across different network architectures and problem geometries. In particular, the inclusion of consistency terms involving gradient jumps was shown to enhance stability in the presence of non-smooth solutions or under-resolved features. Numerical experiments confirm the method’s competitive performance compared to standard collocation-based PINNs, especially in terms of memory usage and training time. These results support the viability of finite element surrogates as a reliable and scalable framework for solving PDEs with neural networks.

\clearpage

\printbibliography
\end{document}